\documentclass[11pt, reqno]{amsart}

\usepackage[margin=1in]{geometry}  
\usepackage{graphicx}               
\usepackage{amsmath}              
\usepackage{amsfonts}              
\usepackage{bm}                        
\usepackage{amsthm}                
\usepackage{mathtools}             
\usepackage{esint}                     
\usepackage{color}
\usepackage{amssymb}             
\usepackage[colorlinks]{hyperref}
\usepackage[noabbrev]{cleveref}
\usepackage{tikz}                       
\usetikzlibrary{patterns}              
\usetikzlibrary{fadings}
\usetikzlibrary{decorations.markings}
\usepackage{subcaption}
\usepackage[english]{babel}
\usepackage[utf8]{inputenc}
\usepackage[T1]{fontenc}
\usepackage{soul}                     
\usepackage{enumitem}
\usepackage[foot]{amsaddr}

\newtheorem{thm}{Theorem}[section]
\newtheorem*{thm*}{Theorem}
\newtheorem{lem}[thm]{Lemma}       
\newtheorem{prop}[thm]{Proposition}
\newtheorem{cor}[thm]{Corollary}

\newtheorem{rmk}[thm]{Remark}
\newtheorem{definition}[thm]{Definition}

\DeclareMathOperator{\loc}{loc}

\DeclareMathOperator{\di}{div}

\newcommand{\RR}{\mathbb{R}}     
      
\newcommand{\NN}{\mathbb{N}}

\newcommand{\Hh}{\mathcal{H}} 
\newcommand{\J}{\mathcal{J}}

\newcommand{\e}{\varepsilon}

\hypersetup{colorlinks = true, linkcolor = blue, citecolor = red}

\makeatletter
\@namedef{subjclassname@2020}{\textup{2020} Mathematics Subject Classification}
\makeatother
         
\numberwithin{equation}{section}

\begin{document}


\title[]{On a class of thin obstacle-type problems for the bi-Laplacian operator}
\author[] {Donatella Danielli$^\dag$} 
\address{$^\dag$School of Mathematics and Statistical Sciences, Arizona State University, Tempe, AZ 85287}
\email {ddanielli@asu.edu} 

\author[] {Giovanni Gravina$^{\ddag,*}$} 
\address{$^\ddag$Department of Mathematics and Statistics, Loyola University Chicago, Chicago, IL 60660}
\email {ggravina@luc.edu} 
\thanks{$^*$ Corresponding author}
\thanks{\emph{Authors ORCID IDs.} Donatella Danielli: \href{https://orcid.org/0000-0002-8659-6671}{0000-0002-8659-6671}, Giovanni Gravina: \href{https://orcid.org/0000-0001-8985-7964}{0000-0001-8985-7964}.}

\date{\today}

\begin{abstract}
This paper investigates the regularity of solutions and structural properties of the free boundary for a class of fourth-order elliptic problems with Neumann-type boundary conditions. The singular and degenerate elliptic operators studied naturally emerge from the extension procedure for higher-order fractional powers of the Laplacian, while the choice of non-linearity considered encompasses two-phase boundary obstacle problems as a special case. After establishing local regularity properties of solutions, Almgren- and Monneau-type monotonicity formulas are derived and utilized to carry out a blow-up analysis and prove a stratification result for the free boundary.
\end{abstract}

\maketitle
\tableofcontents


\section{Introduction}
In this paper, we study the regularity of solutions and some structural properties of the (thin) free boundary for weak solutions of the fourth-order boundary value problem
\begin{equation}
\label{bil}
\left\{
\arraycolsep=1.4pt\def\arraystretch{1.6}
\begin{array}{rll}
\Delta_b^2 u = & 0 & \text{ in } B_1^+, \\
u = & g & \text{ on } S_1^+, \\
\partial_{\nu} u = & \partial_{\nu} g & \text{ on } S_1^+, \\
\partial_y^b u = & 0 & \text{ on } B_1', \\
\partial_y^b \Delta_b u = & f(\cdot, u) &  \text{ on } B_1',
\end{array}
\right.
\end{equation}
where $b \in (-1, 1)$, $\Delta_b$ denotes the elliptic operator
\begin{equation}\label{b_lap}
\Delta_b w \coloneqq y^{-b} \di(y^b \nabla w) = \frac{b}{y} w_y + \Delta w,
\end{equation}
and 
\begin{equation}\label{b_normal}
\partial_y^b w(x, 0) \coloneqq \lim_{y \to 0^+} y^b \partial_y w(x, y)
\end{equation}
is the conormal \emph{interior} derivative. Here $B_r \coloneqq \{X = (x, y) : x \in \RR^n, y \in \RR, \text{ and } |X| < r\}$. Unless otherwise stated, throughout the paper we assume that $n \ge 2$. Moreover, we set $B_r^+ \coloneqq B_r \cap \{y > 0\}$ and use $S_r^+$ and $B_r'$ to denote the lateral and the flat portions of $\partial B_r^+$, respectively. We mention here that $g$ is a given, sufficiently regular boundary datum; the precise assumptions on $g$ are given below in \eqref{H1}, and are motivated by the existence theory for weak solutions to \eqref{bil} (see \Cref{MIN}). Due its technical nature, the definition of weak solutions is given below in \Cref{ws-def} (see also \Cref{weakLP}). Finally, the nonlinearity $f$ is assumed to satisfy regularity and growth conditions that are compatible with the choice 
\begin{equation}
\label{f-def}
f(x, u) \coloneqq \lambda_{-}(u^-)^{p - 1} - \lambda_{+}(u^+)^{p - 1},
\end{equation}
where $\lambda_{\pm} > 0$, $p \ge 1$, and $u^+ \coloneqq \max\{ u, 0 \}$ and $u^- \coloneqq \max \{-u, 0 \}$ denote the positive part and the negative part of $u$, respectively. In particular, a class of two-phase thin obstacle problems can be seen as particular cases of our analysis. This motivates and justifies the assumptions on $f$ that we make throughout the paper (see, for example, \eqref{H2} and \eqref{H3}).

\subsection{Motivation and previous works}
In the seminal paper \cite{MR2354493}, Caffarelli and Silvestre gave a characterization of the fractional Laplacian 
\[
(-\Delta)^s u(x) \coloneqq C(n, s) \operatorname{P.V.} \int_{\RR^n} \frac{u(x) - u(z)}{|x - z|^{n + 2s}}\,dz, \qquad s \in (0, 1),
\]
as the Dirichlet-to-Neumann map for a suitable extension of $u$ in the upper half-space $\RR^{n + 1}_+$. One of the key insights used to establish this result is the equivalence between the $H^s$-seminorm and the Dirichlet energy for the weighted Sobolev space $H^1(\RR^{n + 1}_+; y^{b})$, $b = 1 - 2s$, thus justifying the interest in recent years in the study of elliptic problems driven by the operator $\Delta_b$. For a comprehensive treatment of  the fractional Laplacian see, for example, \cite{MR3967804}, \cite{MR2944369}, \cite{MR4769823}, \cite{MR3916700}, \cite{MR4567945}, and the references therein. While the Caffarelli-Silvestre extension has been employed in a variety of settings, we highlight here some applications that are more directly related to the scope of our paper. A common feature of several of the works described below is the use of Almgren- and Monneau-type monotonicity formulas, made possible by the local character of the extended problem. For instance, in \cite{MR3169789}, Almgren's frequency formula was applied to characterize the asymptotic behavior of solutions to a fractional Laplace equation with a Hardy-type potential, and to establish the strong unique continuation property. In the context of free boundary problems involving the fractional Laplacian, this approach was introduced in the foundational works \cite{MR2405165}, \cite{MR2367025}, and \cite{MR2511747}; see also  \cite{MR2677613}, where the extension was used to study a one-phase free boundary problem for  $(-\Delta)^s$. In \cite{MR2995409},  the authors investigated a two-phase problem with a thin free boundary, which can be viewed as a localized version of the classical two-phase free boundary problem for the half-Laplacian $(-\Delta)^{1/2}$ (i.e., $b = 0$). This result was later extended to all $s \in (0, 1)$ (i.e., $b \in (-1,1)$) in \cite{MR3039830}. Moreover, in \cite{MR3348118},   the non-local analogue of the classical two-phase obstacle problem (or two-phase membrane problem) was studied.

We conclude this brief and by no means exhaustive overview of second-order, extended problems by mentioning the work \cite{DO23}, where Almgren- and Monneau-type monotonicity formulas, together with a blow-up analysis, were employed to provide a classification of the possible vanishing orders for weak solutions to 
\begin{equation}
\label{DO-pb}
\left\{
\arraycolsep=1.4pt\def\arraystretch{1.6}
\begin{array}{rll}
\Delta_b u = & 0 & \text{ in } B_1^+, \\
\partial_y^b u = & \lambda_{+}(u^+)^{p - 1} - \lambda_{-}(u^-)^{p - 1} &  \text{ on } B_1'.
\end{array}
\right.
\end{equation}
Moreover, they proved a stratification result for the nodal set, thus extending previous results of the first author and Jain for the case $b = 0$ (see \cite{MR4144102}). 

Observe that when $f$ is given as in \eqref{f-def}, \eqref{bil} can be interpreted as the fourth-order counterpart to \eqref{DO-pb}. Moreover, \eqref{bil} continues to bear a connection with a lower-dimensional, fractional problem in view of Yang's generalization of the Caffarelli-Silvestre extension for general non-integer orders of the fractional Laplacians (see \cite{yang2013}, and also \cite{MR4429579}). To be precise, when $1 < s < 2$, Yang's result reads as follows.
\begin{thm}[Theorem 3.1 in \cite{yang2013}] 
\label{extension}
For $1 < s < 2$, set $b \coloneqq 3 - 2s$ and let $u \in H^2 (\mathbb{R}^{n+1}_+; y^b)$ be a solution of
\begin{equation*}
\left\{
\begin{array}{rll}
 \Delta_b^2 u = & 0 & \text { in } \mathbb{R}^{n + 1}_+, \\
u = & u_0 & \text { on } \mathbb{R}^n \times \{0\}, \\
\partial_y^b u = & 0 & \text { on } \mathbb{R}^n \times \{0\}, \\
\end{array}
\right.
\end{equation*}
where $u_0 \in H^{s}(\mathbb{R}^n)$. Then we have that
\begin{equation*}
(-\Delta)^{s} u_0(x) = C(n, b) \partial_y^b \Delta_b u(x,0).
\end{equation*}
\end{thm}
The higher-order extension in \Cref{extension} has sparked significant interest and has led to numerous investigations of related problems in recent years. For example, in \cite{MR4169657}, Felli and Ferrero reformulated the corresponding extended problem as a system of two second-order equations and proved that solutions satisfy the strong unique continuation property, as well as the unique continuation from sets of positive measure. This was achieved by establishing an Almgren-type monotonicity formula and through a careful classification of blow-up profiles. A similar analysis was carried out by the same authors in \cite{FF20} for $b = 0$ (that is, for the fractional Laplacian of order $3/2$), where the unique continuation property and the classification of blow-up profiles were shown to hold in the presence of a linear term. In the extended formulation, this corresponds to the choice $f(x, u) \coloneqq h(x)u(x, 0)$, where $h$ is a given function in $C^1(B_1')$. A more detailed discussion of the techniques used in \cite{FF20}, along with a comparison to the strategy developed in this paper for addressing a Neumann coupling term of the form  \eqref{f-def}, is provided in \Cref{AFF-sec}. We also mention \cite{MR4316741}, where the first author and Haj Ali studied the regularity of solutions to \eqref{bil} with $f$ as in \eqref{f-def} and $b = 0$. Additionally, they obtained results on the structure of the free boundary 

\[ 
F(u) \coloneqq (\partial \{ u > 0\} \cup \partial \{ u < 0\}) \cap B_1'
\] 
for $p = 2$ and $p \ge 3$. Since a more in-depth comparison of our results with \cite{MR4316741} involves technical details, the discussion is postponed to \Cref{reg-v-rmk} and \Cref{DHA-comparison}. 

For more details on the (thick) obstacle problem for the bi-Laplacian, we refer the reader to the works of Frehse (see \cite{MR330754} and \cite{MR324208}), Caffarelli and Friedman (see \cite{MR529478}), and Aleksanyan (see \cite{MR4029734}). The regularity of solutions to the thin obstacle problem driven by $\Delta^2$ was explored by Schild (see \cite{MR880399} and \cite{MR752581}). More recently,  versions of the Alt-Caffarelli energy for the bi-harmonic operator were investigated in  \cite{MR4229229},  \cite{MR4026596}, \cite{MR4454382}, and \cite{MR4816111}.

\subsection{Statement of the main results}
Our first main result concerns the regularity of weak solutions to \eqref{bil}. A key step in our proof is an application of a Moser-Trudinger iterative scheme for a reformulation of \eqref{bil} as a system of second-order equations for $u$ and $v \coloneqq \Delta_b u$ (see \Cref{v-solves}). This allows us to show an improvement in integrability for $u, v$, and their traces on $B_R'$ (see \Cref{Lq-reg}), and the result is then combined (when $b \neq 0$) with elliptic regularity theory for operators with singular or degenerate weights in the Muckenhoupt class $A_2$  to establish the regularity of solutions, as stated in the following theorem.

\begin{thm} 
\label{main-reg}
Let $f \colon B_1' \times \RR \to \RR$ be a Carath\'eodory function and assume that there exists a constant $C$ such that  
\[
|f(x, \zeta)| \le C(1 + |\zeta|^{p - 1})
\] 
for $\mathcal{L}^{n}$-a.e.\@ $x \in B_1'$ and all $\zeta \in \RR$, where $p \in [1, p^{\#})$ and $p^{\#} > 2$ is an explicitly computable constant that depends only on $n$ and $b$ $($see \eqref{H4} for the precise definition of $p^{\#}$$)$. Then, if $u$ is a weak solution to \eqref{bil} in the sense of \Cref{ws-def} and $v \coloneqq \Delta_b u$, for all $r < 1$ we have that:
\begin{itemize}
\item[$(i)$] If $b \ge 0$, then $(u,v) \in C^{2, \alpha}(B_r^+\cup B'_r) \times C^{0, \alpha}(B_r^+ \cup B'_r)$ for all $\alpha \in (0, 1 - b)$; 
\item[$(ii)$] If $b < 0$, then $(u,v) \in C^{3, \alpha}(B_r^+\cup B'_r) \times C^{1, \alpha}(B_r^+\cup B'_r)$ for all $\alpha \in (0, - b)$.
\end{itemize}
Additionally, assume that $f$ is locally Lipschitz-continuous in the sense that for all $M > 0$ there exists a constant $C$ such that the inequality
\[
|f(x_1, \zeta_1) - f(x_2, \zeta_2)| \le C (|x_1 - x_2| + |\zeta_1 - \zeta_2|)
\]
holds for all $x_i \in B_1'$ and all $\zeta_i \in \RR$ with $|\zeta_i| \le M$, $i = 1, 2$. Then, for all $r < 1$ we have that: 
\begin{itemize}
\item[$(i)$] If $b \ge 0$, then $\nabla_x' v \in C^{0, \alpha}(B_r^+\cup B'_r)$ for all $\alpha \in (0, 1 - b)$; 
\item[$(ii)$] If $b < 0$, then $\nabla_x' v \in C^{1, \alpha}(B_r^+\cup B'_r)$ for all $\alpha \in (0, - b)$.
\end{itemize}
Here $\nabla_x'$ is used to denote the horizontal components of the gradient. 

Finally, for all $b \in (-1, 1)$, there exists $\beta \in (0, 1)$ such that $\partial_y^b v \in C^{0, \beta}(B_r^+\cup B'_r)$. In particular, if $b = 0$ we have that $(u,v) \in C^{3, \beta}(B_r^+\cup B'_r) \times C^{1, \beta}(B_r^+\cup B'_r)$ for some $\beta \in (0, 1)$.
\end{thm}

For more information on the role played by the bound on the growth condition imposed by $p^{\#}$, we refer the reader to the discussion in \Cref{p-condition}. 

After establishing the regularity of solutions in \Cref{Reg-sec}, the remaining sections of the paper focus on the study of the free boundary. For the case of linear or super-linear growth ($p \ge 2$) and non-degenerate weights ($b \le 0$), we investigate the asymptotic behavior of solutions by carrying out a blow-up analysis. One of our main contributions is the proof of a monotonicity formula for the classical Almgren frequency function 
\[
N_0^{X_0}(r, u, v) \coloneqq \frac{r \int_{B_r^+(X_0)} y^b(|\nabla u|^2 + |\nabla v|^2)\,dX}{\int_{S_r^+(X_0)} y^b(u^2 + v^2)\,d\Hh^n},
\]
where $v$ is again used as a placeholder for $\Delta_b u$ and $X_0 \in B_1'$ (see \Cref{AM-thm}). We remark here that the proof of the frequency formula is non-trivial due to the presence of the non-linear coupling in the Neumann-type boundary condition $\partial_y^b v = f(\cdot, u)$. Our proof draws inspiration from analogous results in previous investigations (see, for example, \cite{MR4316741} and \cite{FF20}). A more detailed discussion of the analytical challenges and the techniques used is provided at the beginning of \Cref{AFF-sec}.

With Almgren's frequency formula at our disposal, we follow the pathway outlined in \cite{MR2511747}. To be precise, after a preliminary blow-up analysis of the so-called Almgren rescalings of $u$ and $v$ (see \Cref{BUL}), we leverage Weiss- and Monneau-type monotonicity formulas to prove the uniqueness of blow-up solutions. Our findings are summarized in the following theorem, where we use $\mathcal{P}_{\mu}$ to denote a certain class of $\mu$-homogeneous, $b$-harmonic functions (see \Cref{Pmu-def} below for the precise definition).

\begin{thm}
\label{BU}
Let $f \colon B_1' \times \RR \to \RR$ be a locally Lipschitz-continuous function and assume that there exists a constant $C$ such that  
\[
|f(x, \zeta)| \le C|\zeta|^{p - 1}
\] 
for all $x \in B_1'$ and all $\zeta \in \RR$, where $p \in [2, p^{\#})$ and $p^{\#}$ is given as in \eqref{H4}. Moreover, assume that $\nabla_x' f$ and $\partial_{\zeta} f$ are Carath\'eodory functions with the property that for every $M > 0$ there exists a constant $C > 0$ such that 
\[
|\nabla_x' f(x, \zeta)| \le C|\zeta| \qquad \text{ and } \qquad |\partial_{\zeta} f(x, \zeta)| \le C
\]
for $\mathcal{L}^n$-a.e. $x \in B_1'$ and all $\zeta \in \RR$ with $|\zeta| \le M$. Finally, let $u$ be a non-trivial weak solution to \eqref{bil} in the sense of \Cref{ws-def}, set $v \coloneqq \Delta_b u$, and assume that $b \le 0$. Then, for every $X_0 \in B_1'$ there exists $\mu \in \mathbb{N} \cup \{0, \infty\}$ such that
\[
\lim_{r \to 0^+} N_0^{X_0}(r, u, v) = \mu.
\] 
Additionally, if $\mu < \infty$, there exist $\tilde{u}_{\mu}^{X_0}, \tilde{v}_{\mu}^{X_0} \in \mathcal{P}_{\mu}$ such that 
\begin{align*} 
\frac{u(X_0 + rX)}{r^{\mu}} \to \tilde{u}_{\mu}^{X_0}, \quad \frac{v(X_0 + rX)}{r^{\mu}} & \to \tilde{v}_{\mu}^{X_0} \qquad \text{ in } H^1(B_R^+; y^b), \\
\frac{u(X_0 + rX)}{r^{\mu}} \to \tilde{u}_{\mu}^{X_0}, \quad \frac{v(X_0 + rX)}{r^{\mu}} & \to \tilde{v}_{\mu}^{X_0} \qquad \text{ in } C^{1, \alpha}(B_R^+),
\end{align*}
for some $\alpha \in (0, 1)$ and all $R > 0$ as $r \to 0^+$. Moreover, $\tilde{u}_{\mu}^{X_0}$ and $\tilde{v}_{\mu}^{X_0}$ cannot both vanish identically. 
\end{thm}
Finally, the free boundary
\[
F(u) \coloneqq \left(\partial\{ u > 0 \} \cup \partial\{ u < 0\}\right) \cap B_1'
\]
is analyzed by dividing it into two components: the regular part $R(u)$, which inherits the regularity of the solution, and the singular part $S(u)$. Thanks to a by now standard argument based on the continuous dependence of blow-up solutions with respect to the singular point $X_0$ and Whitney's extension theorem, for $n = 2, 3$ we prove that the subset of $S(u)$ consisting of points with frequency $\mu \in \NN$ is contained in the countable union of $C^1$ regular manifolds of dimensions $d = 0, 1, \dots, n - 1$. For $n \ge 4$, the result holds for frequencies that are sufficiently small (see \Cref{GPthm}). 

\subsection{Outline of the paper}
The paper is organized as follows. In \Cref{Prem-sec} we introduce our notations and recall well-known results such as Poincar\'e, trace, and Sobolev-type inequalities for weighted Sobolev spaces, as well as a Rellich-type identity. Additionally, we collect results on the regularity of solutions to singular or degenerate elliptic equations that will be useful throughout the paper. In \Cref{Ex-sec}, we give the definition of weak solution and formulate assumptions under which existence of solutions can be established via variational methods (see \Cref{MIN}). The regularity of weak solutions is addressed in \Cref{Reg-sec}. Our main result in this section is the proof of \Cref{main-reg}. With this at hand, the rest of the paper is devoted to the study of the free boundary. To be precise, in \Cref{AFF-sec} we establish the almost-monotonicity of the frequency functional, while in \Cref{ABU-sec} we begin our investigation of blow-up solutions by proving the existence of Almgren-type blow-ups. In \Cref{MW-sec}, we derive monotonicity formulas of Monneau- and Weiss-type that allow us to prove the non-degeneracy of solutions. This is used later in the section to prove existence and uniqueness of homogenous blow-up solutions, thus concluding the proof of \Cref{BU}. Finally, in \Cref{FB-sec} we investigate the structure of the free boundary.

\section{Preliminary results}
\label{Prem-sec}
For the convenience of the reader, in this section we collect some of the definitions and tools used throughout the paper.
\subsection{Weighted Sobolev Spaces}
The treatment of \eqref{bil} is facilitated by the introduction of Lebesgue and Sobolev spaces with weight $\omega(X) \coloneqq y^b$, $b \in (-1, 1)$. To be precise, for a given open, bounded domain $\Omega \subset \RR^{n + 1}$ with Lipschitz boundary,  we define 
\[
L^2(\Omega; y^b) \coloneqq \{ w : |y|^{b/2}w \in L^2(\Omega) \}
\]
and 
\[
H^k(\Omega; y^b) \coloneqq \{w \in L^2(\Omega; y^b) : \partial^{\alpha}w \in L^2(\Omega; y^b) \text{ for all } \alpha \in \NN_0^{n + 1} \text{ with } |\alpha| \le k\}.
\]
Other function spaces are defined analogously. For more information, we refer the reader to the monograph \cite{MR0802206}. 

Observe that the assumption $b \in (-1, 1)$ guarantees that $|y|^b$ belongs to the Muckenhoupt class $A_2$, that is, there exists a positive constant $C$ such that for every ball $B \subset \RR^{n + 1}$ we have that
\[
\left(\fint_B |y|^b\,dX\right)\left(\fint_B |y|^{-b}\,dX\right) \le C.
\]
Moreover, we remark that the weight is singular when $b < 0$ and degenerate for $b > 0$.

This subsection collects widely known results such as Poincar\'e, trace, and Sobolev-type inequalities for the space $H^1(B_r^+; y^b)$. We conclude the subsection with a Rellich-type identity, which will be instrumental for the proof of the frequency formula later in \Cref{AFF-sec}. The identity is obtained under certain regularity assumptions that are standard in the framework we consider and it can thus be verified for our solutions (see, in particular, \Cref{FS-reg} and \Cref{vH2} below).

\begin{lem}[Poincar\'e-type inequality]
\label{PI}
For $w \in H^1(B_r^+; y^b)$ we have
\[
\frac{n  + b}{r^2} \int_{B_r^+} y^bw^2\,dX \le \frac{1}{r} \int_{S_r^+} y^bw^2\,d\Hh^n + \int_{B_r^+} y^b|\nabla w|^2\,dX.
\]
\end{lem}
\begin{proof}
The proof is a straightforward adaptation of a well-known argument, and is included here for the reader's convenience. Recalling that $X \coloneqq (x, y)$, we have that
\begin{align*}
(n + 1) \int_{B_r^+} y^bw^2\,dX & = \int_{B_r^+} (\di(y^bw^2 X) - \nabla(y^b w^2) \cdot X)\,dX \\
& = r \int_{S_r^+} y^bw^2\,d\Hh^n - \int_{B_r^+} (2y^b w \nabla w \cdot X + b y^b w^2)\,dX \\
& \le r \int_{S_r^+} y^bw^2\,d\Hh^n + \int_{B_r^+} ((1 - b)y^b w^2 + r^2 y^b|\nabla w|^2)\,dX.
\end{align*} 
The desired estimate follows by rearranging the terms in the previous inequality.
\end{proof}

\begin{lem}[Trace operators]
\label{TO}
There exist compact trace operators
\begin{align*}
\operatorname{Tr} & \colon H^1(B_r^+; y^b) \to L^2(S_r^+; y^b), \\
\operatorname{Tr} & \colon H^1(B_r^+; y^b) \to L^2(B_r').
\end{align*}
Moreover, 
\[
\operatorname{Tr} \colon H^1(B_r^+; y^b) \to H^{(1 - b)/2}(B_r')
\]
is linear, continuous, and onto. 
\end{lem}
For a proof of \Cref{TO}, we refer to \cite{MR3169789} and Proposition 2.1 in \cite{MR3226738}. In order to ease notation, in the following we continue to use $w$ in place of $\operatorname{Tr}w$ whenever $w \in H^1(B_r^+; y^b)$.

\begin{lem}[Trace inequality]
\label{TI}
There exists a constant $C$ such that for all $r > 0$ and all $w \in H^1(B_r^+; y^b)$, we have that
\[
\int_{B_r'}w^2\,dx \le C \left(r^{1 - b} \int_{B_r^+} y^b |\nabla w|^2\,dX + r^{-b} \int_{S_r^+} y^b w^2\,d\Hh^n \right).
\]
\end{lem}
This result follows immediately from Lemma 2.5 in \cite{MR3169789} once one observes that $|x|<r$ on $B_r'$.

\begin{lem}[Sobolev-type inequalities]
\label{SII}
There exist constants $S = S(n, b)$ and $S'(n, b)$ such that for all $w \in H^1(B_1^+; y^b)$, we have that
\begin{align}
\left(\int_{B_1'} |w|^{2^*_b}\,dx \right)^{2/2^*_b} & \le S' \left(\int_{B_1^+} y^b |\nabla w|^2\,dX + \int_{S_1^+} y^b w^2\,d\Hh^n\right), \label{S'E} \\
\left(\int_{B_1^+}y^b |w|^{2^{**}_b}\,dX\right)^{2/2^{**}_b} & \le S \left(\int_{B_1^+} y^b |\nabla w|^2\,dX + \int_{B_1^+} y^b w^2\,dX \right). \label{SE}
\end{align}
Here $2^*_b$ and $2^{**}_b$ are the fractional Sobolev exponent, defined via
\[
2^*_b \coloneqq \frac{2n}{n - 1 + b}
\]
and 
\[
2^{**}_b \coloneqq 
\left\{
\arraycolsep=2.4pt\def\arraystretch{2.6}
\begin{array}{rl}
\displaystyle \frac{2(n + 1 + b)}{n - 1 + b} & \text{ if } 0 < b < 1, \\
\displaystyle \frac{2(n + 1)}{n - 1} & \text{ if } -1 < b \le 0.
\end{array}
\right.
\]
\end{lem}
For a proof of \eqref{S'E}, we refer the reader to Lemma 2.6 in \cite{MR3169789}, whereas \eqref{SE} follows from Theorem 19.10 in \cite{MR1069756}. We also remark that versions of \eqref{S'E} and \eqref{SE} continue to hold in $B_r^+$ in view of a straightforward scaling argument. 

In the following, we let 
\[
H^2_x(B_r^+; y^b) \coloneqq \left\{w \in H^1(B_r^+; y^b) : \partial_{x_i} w \in H^1(B_r^+; y^b), i = 1, \dots, n \right\}.
\]

\begin{lem}[Rellich identity]
\label{Rellich}
Let $w \in H^2_x(B_R^+; y^b)$ be such that $\partial_y^b w \in H^1(B_R^+; y^{-b})$ (see \eqref{b_normal} for the definition of $\partial_y^b w$). Then
\begin{multline}
\label{RID}
r \int_{S_r^+} y^b \left(|\nabla w|^2 - 2 \left(\nabla w \cdot \frac{X}{r}\right)^2\right)\,d\Hh^n = (n + b - 1) \int_{B_r^+} y^b |\nabla w|^2\,dX \\ - 2 \int_{B_r^+} y^b (\nabla w \cdot X) \Delta_b w\,dX - 2 \int_{B_r'} (\nabla_x' w \cdot x)\partial_y^bw\,dx
\end{multline}
for $\mathcal{L}^1$-a.e.\@ $r \in (0, R)$.
\end{lem}

\begin{proof}
Observe that the all the terms on the right-hand side of \eqref{RID} are well-defined in view of the regularity assumptions on $w$. Notice also that, by the coarea formula, 
\[
\int_{B_R^+} y^b\left||\nabla w|^2 - 2 \left(\nabla w \cdot \frac{X}{|X|}\right)^2\right|\,dX = \int_0^R \int_{S_r^+} y^b\left||\nabla w|^2 - 2 \left(\nabla w \cdot \frac{X}{r}\right)^2\right|\,d\Hh^n\,dr.
\]
Consequently, if we set
\[
d(r) \coloneqq \int_{S_r^+} y^b \left(|\nabla w|^2 - 2 \left(\nabla w \cdot \frac{X}{r}\right)^2\right)\,d\Hh^n,
\]
we obtain that $d \in L^1(0, R)$. This implies that the left-hand side of \eqref{RID} is well-defined for $\mathcal{L}^1$-a.e.\@ $r \in (0, R)$. 

By a standard density argument, it is enough to establish \eqref{RID} for $w \in C^{\infty}(B_r^+)$. Observe that
\begin{align*} 
\di(X y^b|\nabla w|^2) & = X \cdot \nabla(y^b |\nabla w|^2) + (n + 1) y^b |\nabla w|^2 \\
& = 2 y^b\sum_{i, k} X_i(\partial_{i k} w)(\partial_k w) + (n + b + 1) y^b|\nabla w|^2
\end{align*}
and 
\begin{align*}
\di((\nabla w \cdot X)y^b\nabla w) & = \nabla(\nabla w \cdot X) \cdot y^b \nabla w + y^b (\nabla w \cdot X) \Delta_b w \\
& = y^b|\nabla w|^2 + y^b\sum_{i, k} X_i(\partial_{i k} w)(\partial_k w) + y^b(\nabla w \cdot X)\Delta_b w.
\end{align*}
Combining the two previous identities, we obtain that 
\[
\di(X y^b|\nabla w|^2 - 2(\nabla w \cdot X)y^b\nabla w) = (n + b - 1) y^b |\nabla w|^2 - 2y^b(\nabla w \cdot X)\Delta_b w.
\]
From this, the desired identity follows by an application of the divergence theorem in $B_r^+ \cap \{y > \tau\}$ and by subsequently letting $\tau \to 0^+$.
\end{proof}

\subsection{Regularity of solutions to singular or degenerate elliptic equations}
\label{reg-subsec}
This subsection contains a discussion of regularity results concerning weak solutions for a class of second-order elliptic problems in divergence form, defined via 
\[
\mathcal{L}_b(w) \coloneqq \di(y^b\nabla w), \qquad b \in (-1, 1).
\]
The regularity theory for degenerate or singular elliptic (and parabolic) equations has received considerable attention in recent years due to its connection to the fractional Laplacian and its realization as a Dirichlet-to-Neumann map (see \cite{MR2354493}, \cite{yang2013}, and \cite{MR4429579}). We present the results in a framework pertinent to our paper's scope, while acknowledging that broader formulations exist (see, for example, \cite{MR0643158}, \cite{MR4536419}, \cite{MR4207950}, and the references therein).

Let $\gamma_1 \in L^2(B_R^+; y^b)$ and $\gamma_2 \in L^2(B_R')$ be given, and consider the Neumann problem
\begin{equation}
\label{2nd}
\left\{
\arraycolsep=1.4pt\def\arraystretch{1.6}
\begin{array}{rll}
\mathcal{L}_b(w) = & y^b \gamma_1 & \text{ in } B_R^+, \\
\partial_y^b w = & \gamma_2  & \text{ on } B_R'.
\end{array}
\right.
\end{equation}
We recall that $\partial_y^b w(X) \coloneqq y^b \partial_y w(X)$ for $X = (x,y) \in B_R^+$ and $\partial_y^b w(x, 0)$ is used to denote its trace on $B_R'$, that is, $\partial_y^b w(x, 0) \coloneqq \lim_{y \to 0^+} y^b \partial_y w(x, y)$. 

\begin{definition}
\label{ws2-def}
We say that $w \in H^1(B_R^+; y^b)$ is a weak solution to \eqref{2nd} if 
\[
\int_{B_R^+} y^b \nabla w \cdot \nabla \phi\,dX + \int_{B_R^+} y^b \gamma_1 \phi\,dX + \int_{B_R'} \gamma_2 \phi\,dx = 0
\]
for all $\phi \in H^1_{0, S_R^+}(B_R^+; y^b)$, where 
\[
H^1_{0, S_R^+}(B_R^+; y^b) \coloneqq \overline{C^{\infty}_c(\overline{B_R^+} \setminus \overline{S_R^+})}^{\| \cdot \|_{H^1(B_R^+; y^b)}}.\]
\end{definition}

We begin with a general Sobolev-type regularity result for weak solutions to \eqref{2nd}. For a proof (as well as for a more general statement), we refer the reader to Theorem 2.1 in \cite{MR4536419}. 
\begin{thm}
\label{FS-reg}
Under the assumptions that $\gamma_1 \in L^2(B_R^+; y^b)$ and $\gamma_2 \in W^{1, \frac{2n}{n + 1 - b}}(B_R')$, let $w$ be a weak solution to \eqref{2nd} in the sense of \Cref{ws2-def}. Then, for all $r \in (0, R)$ we have that 
\[
w \in H^2_x(B_r^+; y^b) \qquad \text{ and } \qquad \partial_y^b w \in H^1(B_r^+; y^{-b}).
\]
Moreover, there exists a constant $C = C(n, b, R, r)$ such that 
\[
\|\nabla_x' w\|_{H^1(B_r^+; y^b)} + \| \partial_y^b w\|_{H^1(B_r^+; y^{-b})} \le C \left(\|w\|_{H^1(B_R^+; y^b)} + \|\gamma_1\|_{L^2(B_R^+; y^b)} + \|\gamma_2\|_{W^{1, \frac{2n}{n + 1 - b}}(B_R')}\right).
\] 
\end{thm}

Observe that weak solutions as in \Cref{FS-reg} locally satisfy the assumptions of \Cref{Rellich} (see also Proposition 2.3 in \cite{MR4536419}).

For the H\"older continuity of solutions, our reference is \cite{MR4207950}, where Schauder-type estimates are obtained for a large class of degenerate or singular elliptic problems. The first result we recall concerns $\mathcal{L}_b$-harmonic functions subject to non-homogenous Neumann boundary conditions. For a proof, see Theorem 1.5 and Theorem 1.6 in \cite{MR4207950}. 

\begin{thm}
\label{STV1.5}
Let $w \in H^1(B_R^+; y^b)$ be a weak solution to \eqref{2nd} with $\gamma_1 = 0$ and $\gamma_2 \in L^q(B_R')$ with $q > n/(1 - b)$. Then, for all $r \in (0, R)$ and $\alpha \in (0, 1 - b - n/q]$ there exists a constant $C = C(n, b, R, r, q, \alpha)$ such that 
\[
\|w\|_{C^{0, \alpha}(B_r^+\cup B'_r)} \le C \left(\|w\|_{L^2(B_R^+; y^b)} + \|\gamma_2\|_{L^q(B_R')}\right).
\]
Moreover, if $b < 0$ and $\gamma_2 \in L^q(B_R')$ with $q > -n/b$ then, for all $r \in (0, R)$ and $\alpha \in (0, - b - n/q]$, there exists a constant $C = C(n, b, r, R, q, \alpha)$ such that 
\[
\|w\|_{C^{1, \alpha}(B_r^+\cup B'_r)} \le C \left(\|w\|_{L^2(B_R^+; y^b)} + \|\gamma_2\|_{L^q(B_R')}\right).
\]
\end{thm}

We mention here that H\"older estimates for the conormal derivative can be obtained under additional assumptions on $w$ and $\gamma_2$. We report below the precise statement that we use in the following sections, and refer the reader to Lemma 4.5 in \cite{MR3165278} for the proof. 

\begin{lem}
\label{4.5CS}
Let $w \in L^{\infty}(B_R^+) \cap H^1(B_R^+; y^b)$ be a weak solution to \eqref{2nd} with $\gamma_1 = 0$ and $\gamma_2 \in C^{0, \sigma}(B_R')$ with $\sigma \in (0, 1)$. Then, for all $r \in (0, R)$ there exists $\beta \in (0, \sigma)$, $\beta = \beta(n, b, \sigma)$, such that $\partial_y^b w \in C^{0, \beta}(B_r^+{\cup B'_r})$. Moreover, 
\[
\|\partial_y^b w\|_{C^{0, \beta}({B_r^+\cup B'_r})} \le C,
\]
where $C$ is a constant that depends only on $n, b, r, R, \|w\|_{L^{\infty}(B_R^+)}$, and $\| \gamma_2 \|_{C^{0, \sigma}(B_R')}$.
\end{lem}

Next, let $w \in H^1(B_R^+; y^b)$ be a weak solution to \eqref{2nd} with $\gamma_2 = 0$. Then, if $w^* \in H^1(B_R; |y|^b)$ and $\gamma_1^* \in L^2(B_R; |y|^b)$ denote the extensions of $w$ and $\gamma_1$, respectively, obtained by an even reflection across the hyperplane $\{y = 0\}$, we have that $w^*$ is a weak solution to $\mathcal{L}_b(w^*) = |y|^b \gamma_1^*$ in $B_R$, in the sense that 
\begin{equation}
\label{weak-even}
\int_{B_R} |y|^b \nabla w^* \cdot \nabla \phi\,dX + \int_{B_R} |y|^b \gamma_1^* \phi\,dX = 0
\end{equation}
for all $\phi \in H^1(B_R; |y|^b)$ with $\phi = 0$ on $\partial B_R$. We say that $h \in L^2(B_R;|y|^b)$ is even in $y$ if $h(x,y) = h(x, -y)$ for $\mathcal{L}^{n + 1}$-a.e.\@ $X \in B_R$.

\begin{thm}
\label{STV}
For $k \in \NN \cup \{0\}$ and $\alpha \in (0, 1)$, let $\gamma \in L^2(B_R, |y|^b)$ and $w \in H^1(B_R; |y|^b)$ be even in $y$ and assume that $w$ is a weak solution to $\mathcal{L}_b(w) = |y|^b \gamma$ in $B_R$ in the sense of \eqref{weak-even}. Then, for every $r \in (0, R)$, the following hold:
\begin{itemize}
\item[$(i)$] If $\gamma \in C^{k, \alpha}(B_R)$ then $w \in C^{k + 2, \alpha}(B_r)$. If moreover $\gamma \in C^{\infty}(B_R)$, then $w \in C^{\infty}(B_R)$.
\item[$(ii)$] If $\gamma \in L^q(B_R; |y|^b)$ with $q > (n + 1 + b^+)/2$ and $\alpha \in (0, 1) \cap (0, 2 - (n + 1 + b^+)/q]$ then $w \in C^{0, \alpha}(B_r)$ and there exists a constant $C = C(n, b, R, r, q, \alpha)$ such that 
\[
\|w\|_{C^{0, \alpha}(B_r)} \le C \left(\|w\|_{L^2(B_R; |y|^b)} + \|\gamma\|_{L^q(B_R; |y|^b)}\right).
\]
\item[$(iii)$] If $\gamma \in L^q(B_R; |y|^b)$ with $q > n + 1 + b^+$ and $\alpha \in (0, 1 - (n + 1 + b^+)/q]$ then $w \in C^{1, \alpha}(B_r)$ and there exists a constant $C = C(n, b, R, r, q, \alpha)$ such that 
\[
\|w\|_{C^{1, \alpha}(B_r)} \le C \left(\|w\|_{L^2(B_R; |y|^b)} + \|\gamma\|_{L^q(B_R; |y|^b)}\right).
\]
\end{itemize}
\end{thm}
For a proof of \Cref{STV} (including the case of more degenerate weights $|y|^b$, $b \ge 1$), see Theorem 1.1 and Theorem 1.2 in \cite{MR4207950}. The H\"older continuity of solutions for general $A_2$ weights was previously obtained in the pioneering work \cite{MR0643158}.

\section{Existence of weak solutions}
\label{Ex-sec}
We begin by introducing the assumptions utilized throughout the section. Following this, we give the definition of a weak solution to \eqref{bil} and proceed to establish existence by employing the direct method in the calculus of variations.

In the following we let, for $\Delta_b w$ as in \eqref{b_lap},
\[
\Hh \coloneqq \{w \in H^1(B_1^+; y^b) : \Delta_b w \in L^2(B_1^+; y^b)\}
\]
and set 
\[
\| w \|_{\Hh} \coloneqq \left(\int_{B_1^+} y^b\left(|w|^2 + |\nabla w|^2 + |\Delta_b w|^2\right)\,dX\right)^{1/2}.
\]
Let $p \ge 1$. Next, we specify the assumptions on the boundary datum $g$ and the nonlinearity $f$.
\begin{enumerate}[label=(H.\arabic*), ref=H.\arabic*]
\item \label{H1} Assume that $g \in \Hh \cap L^p(B_1')$ satisfies $\partial_y^b g = 0$ $B_1'$.
\item \label{H2} Assume that $f \colon B_1' \times \RR \to \RR$ is Carath\'eodory and satisfies 
\[
|f(x, \zeta)| \le C(1 + |\zeta|^{p - 1})
\] 
for $\mathcal{L}^{n}$-a.e.\@ $x \in B_1'$ and all $\zeta \in \RR$.
\item \label{H3} Let 
\[
F(x, \zeta) \coloneqq -2 \int_0^\zeta f(x, \sigma)d\sigma.
\]
Assume 
that there is a positive constant $c$ such that
\[
F(x, \zeta) \ge c|\zeta|^p
\]
for $\mathcal{L}^{n}$-a.e.\@ $x \in B_1'$ and all $\zeta \in \RR$.
\end{enumerate}

\begin{definition}
\label{ws-def}
We say that $u \in \Hh \cap L^p(B_1')$ is a weak solution of 
\begin{equation}
\label{bil-noBC}
\left\{
\arraycolsep=1.4pt\def\arraystretch{1.6}
\begin{array}{rll}
\Delta_b^2 u = & 0 & \text{ in } B_1^+, \\
\partial_y^b u = & 0 & \text{ on } B_1', \\
\partial_y^b \Delta_b u = & f(\cdot, u) &  \text{ on } B_1',
\end{array}
\right.
\end{equation} 
if
\begin{equation}
\label{wpde}
\int_{B_1^+} y^b \Delta_b u \Delta_b \phi\,dX = \int_{B_1'} f(x, u) \phi\,dx
\end{equation}
for all $\phi \in \Hh_{0, S_1^+} \cap L^p(B_1')$, where
\[
\Hh_{0, S_1^+} \coloneqq \overline{\left\{\phi \in C^{\infty}_c(\overline{B_1^+} \setminus \overline{S_1^+}) : \partial_y^b \phi = 0 \text{ on } B_1' \right\}}^{\| \cdot \|_{\Hh}}.
\] 
\end{definition}

\begin{rmk}
\label{weakLP}
\begin{itemize}
\item[$(i)$] Observe that the integral on the right-hand side of \eqref{wpde} is well-defined for all admissible test functions. Indeed, in view of \eqref{H2}, an application of H\"older's inequality yields
\[
\left|\int_{B_1'} f(x, u) \phi\,dx\right| \le C\left(1 + \|u\|_{L^p(B_1')}^{p - 1}\right) \|\phi\|_{L^p(B_1')}.
\]
\item[$(ii)$] We say that $u$ solves \eqref{bil} if in addition to being a solution to \eqref{bil-noBC} in the sense of \Cref{ws-def}, it satisfies $u = g$ and $\partial_{\nu} u = \partial_{\nu} g$ on $S_1^+$ in the sense of traces.
\end{itemize}
\end{rmk}

Next, we consider the associated energy
\begin{equation}
\label{J-def}
\J(w) \coloneqq \int_{B_1^+} y^b|\Delta_b w|^2\,dX +  \int_{B_1'} F(x, w)\,dx, 
\end{equation}
defined over the class
\begin{equation}
\label{A-def}
\mathcal{A} \coloneqq \{w \in \Hh : w = g \text{ on } S_1^+, \partial_{\nu}w = \partial_{\nu}g \text{ on } S_1^+, \text{ and } \partial_y^b w = 0 \text{ on } B_1'\}.
\end{equation} 

\begin{thm}
\label{MIN}
Under the assumptions \eqref{H1}--\eqref{H3}, let $\J$ and $\mathcal{A}$ be given as in \eqref{J-def} and \eqref{A-def}, respectively. Then $\J$ admits a global minimizer $u \in \mathcal{A}$. Moreover, $u$ is a weak solution to \eqref{bil} in the sense of \Cref{ws-def} and \Cref{weakLP} $(ii)$.
\end{thm}
\begin{proof}
The proof follows by standard variational methods. We include here the details for the reader's convenience. 
\newline
\textbf{Step 1:} Observe that \eqref{H1} implies that $g \in \mathcal{A}$. Moreover, using the fact that $F(x, \zeta) \le C(1 + |\zeta|^p)$ (see \eqref{H2} and \eqref{H3}), we can also conclude that $\J(g) < \infty$. Let $\{w_j\}_j \subset \mathcal{A}$ be an infimizing sequence. Then, for every $j$, by \eqref{PI} we have that 
\begin{equation}
\label{poinc-1}
(n + b) \int_{B_1^+} y^b|w_j - w_1|^2\,dX \le \int_{B_1^+} y^b |\nabla (w_j - w_1)|^2\,dX.
\end{equation}
By the divergence theorem and H\"older's inequality, we obtain that
\begin{align}
\int_{B_1^+} y^b |\nabla (w_j - w_1)|^2\,dX & = \int_{B_1^+} y^b(w_j - w_1)\Delta_b(w_j - w_1)\,dX \notag \\ 
& \le \left(\int_{B_1^+}y^b|w_j - w_1|^2\,dX\right)^{1/2} \left(\int_{B_1^+} y^b |\Delta_b(w_j - w_1)|^2\,dX\right)^{1/2}. \label{grad-1-est}
\end{align}
Combining \eqref{poinc-1} and \eqref{grad-1-est} we conclude that 
\[
\|w_j - w_1\|_{H^1(B_1^+; y^b)}^2 \le C(n, b) \int_{B_1^+}y^b|\Delta_b(w_j - w_1)|^2\,dX \le C(n, b) \J(g),
\]
where the last inequality follows by recalling that $F$ is nonnegative and by noticing that, without loss of generality, we can assume that $\J(w_j) \le \J(g)$ for all $j$. Consequently, there exists a constant $C = C(n, b)$ such that
\[
\| w_j - w_1\|_{\Hh} \le C \J(g)^{1/2}.
\]
By standard compactness arguments together with \Cref{TO}, eventually extracting a subsequence (which we do not relabel), we can find $u \in \mathcal{A}$ such that $w_j \rightharpoonup u$ in $H^1(B_1^+; y^b)$, $\Delta_b w_j \rightharpoonup \Delta_b u$ in $L^2(B_1^+; y^b)$, and $w_j \to u$ in $L^2(B_1')$ and pointwise a.e.\@ on $B_1'$. In particular, we also obtain that $F(\cdot, w_j(\cdot)) \to F(\cdot, u(\cdot))$ pointwise a.e.\@ on $B_1'$. Since by Fatou's lemma $\J$ is weakly lower-semicontinuous with respect to the combination of these notions of convergence (see \eqref{H3}), this shows that $u$ is a global minimizer for $\J$ over $\mathcal{A}$. 
\newline
\textbf{Step 2:} We now proceed to show that $u$ is a weak solution to \eqref{bil}. To this end, observe that by \eqref{H3} we have that
\[
c\|u\|_{L^p(B_1')}^p \le \J(u) < \infty,
\]
thus proving that $u \in L^p(B_1')$. 
Finally, as one can readily check, for every $\phi \in C^{\infty}_c(\overline{B_1^+} \setminus \overline{S_1^+})$ with $\partial_y^b \phi = 0$ on $B_1'$, we have that $u + \e \phi \in \mathcal{A}$. For $\e > 0$, the minimality of $u$ implies that
\begin{equation}
\label{DQ-EL}
0 \le \frac{\J(u + \e \phi) - \J(u)}{\e} = \int_{B_1^+} y^b(2 \Delta_b u \Delta_b \phi + \e |\Delta_b \phi|^2)\,dX + \int_{B_1'} \frac{F(x, u + \e \phi) - F(x, u)}{\e}\,dx.
\end{equation}
By Lebesgue's dominated convergence theorem, letting $\e \to 0^+$ in the previous inequality shows that 
\[
0 \le \int_{B_1^+} y^b\Delta_b u \Delta_b \phi\,dX - \int_{B_1'}f(x, u)\phi\,dx.
\]
To conclude the proof, it is enough to observe that repeating the argument with $-\phi$ yields the opposite inequality.
\end{proof}

We remark that if in addition we have that $\zeta \mapsto F(x,\zeta)$ is convex for $\mathcal{L}^{n}$-a.e.\@ $x \in B_1'$, then $\J$ is strictly convex over the closed, convex set $\mathcal{A}$. This, in turn, implies that the minimizer of \Cref{MIN} is unique. Notice that when $f$ is given as in \eqref{f-def}, one has that
\[
F(x, \zeta) = \frac{2}{p} (\lambda_{-}(\zeta^-)^{p} + \lambda_{+}(\zeta^+)^{p}),
\]
which is indeed convex with respect to $\zeta$.

\section{Regularity of solutions}
\label{Reg-sec}
This section is dedicated to proving regularity results for $u$ and $\Delta_b u$. We begin by showing that the fourth-order problem \eqref{bil-noBC} can be recast as a system of coupled second-order equations for the operator $\mathcal{L}_b$ (see \Cref{v-solves}). Using this reformulation, we employ a combination of a Moser--Trudinger iteration scheme and a bootstrapping argument to establish our main regularity result, as outlined in \Cref{Lq-reg}, \Cref{Reg-cor} and \Cref{improved-reg}.

\begin{lem}
\label{aux-2nd}
For every $\psi \in C^{\infty}_c(B_1^+)$ and $R < 1$, there exists $\Psi \in \Hh \cap C^{2, \alpha}(B_R^+)$ such that
\[
\left\{
\arraycolsep=1.4pt\def\arraystretch{1.6}
\begin{array}{rll}
\Delta_b \Psi = & \psi & \text{ in } B_1^+, \\
\Psi = & 0 & \text{ on } S_1^+, \\
\partial_y^b \Psi = & 0 & \text{ on } B_1'.
\end{array}
\right.
\]
\end{lem}
\begin{proof}
As one can readily check, the existence of a weak solution follows from a similar (but simpler) argument to that in the previous section; therefore, we omit the details. The regularity of $\Psi$ is a direct consequence of \Cref{STV}.
\end{proof}

\begin{thm}
\label{LM}
Let $u$ be a weak solution to \eqref{bil-noBC} in the sense of \Cref{ws-def}. Then $\Delta_b u \in H^1(B_r^+; y^b)$ for $r < 1$.
\end{thm}
\begin{proof}
The proof uses arguments that are well-known to the experts (see, for example, Theorem 2.5 in \cite{MR4316741} and Proposition 2.4 in \cite{MR4169657}). The details are included here for the reader's convenience. 
\newline 
\textbf{Step 1:} We begin by considering a smooth cut-off function $\eta \in C^{\infty}_{c}(B_1)$ such that $\eta \equiv 1$ in $B_r$ and $\partial_y^b \eta = 0$ on $B_1'$, and a function $\Psi \in \Hh \cap C^{2, \alpha}(B_R^+)$ with $\Psi = 0$ on $S_1^+$, where $r < R < 1$ is such that $\operatorname{supp}(\eta) \subset B_R$. It is then straightforward to verify that $\eta \Psi \in L^p(B_1')$ and conclude that $\eta \Psi$ is an admissible test function for \eqref{wpde} for any such function $\Psi$. Consequently, if we set $w \coloneqq \eta \Delta_b u$, we obtain that
\begin{align}
\int_{B_1^+} y^b w \Delta_b \Psi\,dX 
& = - \int_{B_1^+}  y^b(2 \nabla \eta \cdot \nabla \Psi + \Psi \Delta_b \eta)\Delta_b u\,dX + \int_{B_1'} f(x, u)\eta \Psi \,dx. \label{weakPDEreg}
\end{align}
Next, we claim that $w$ is the only solution to \eqref{weakPDEreg} that vanishes on $S_1^+$. Indeed, if $w_1, w_2$ are two such solutions, then $w_0 \coloneqq w_1 - w_2$ satisfies
\[
\int_{B_1^+} y^b w_0 \Delta_b \Psi\,dX = 0
\]
for all $\Psi$ as above. In turn, \Cref{aux-2nd} implies that 
\[
\int_{B_1^+}y^b w_0 \psi \,dX = 0
\]
for all $\psi \in C^{\infty}_c(B_1^+)$, and the claim readily follows.
\newline
\textbf{Step 2:} To conclude, we observe that, for fixed $\eta$ and $u$, the right-hand side of \eqref{weakPDEreg} defines a bounded linear operator in $\Psi$ on the space
\[
V \coloneqq \left\{ \Psi \in H^1_{0, S_1^+}(B_1^+; y^b) : \eta \Psi \in L^p(B_1')\right\}, 
\]
equipped with the norm 
\[
\|\Psi\|_V \coloneqq \| \nabla \Psi \|_{L^2(B_1^+; y^b)} + \|\eta \Psi\|_{L^p(B_1')}.
\]
Then, an application of the Lions--Lax--Milgram theorem yields the existence of a function $\xi \in H^1_{0, S_1^+}(B_1^+; y^b)$ such that 
\[
\int_{B_1^+} y^b \nabla \xi \cdot \nabla \Psi\,dX = \int_{B_1^+}  y^b(2 \nabla \eta \cdot \nabla \Psi + \Psi \Delta_b \eta)\Delta_b u\,dX - \int_{B_1'} f(x, u)\eta \Psi \,dx
\]
for all $\Psi \in V$. Using the fact that the space of all functions in $\Hh \cap C^{2, \alpha}(B_R^+)$ that vanish on $S_1^+$ is a subspace of $V$ for $R < 1$ sufficiently large, the divergence theorem can be used to show that $\xi$ satisfies \eqref{weakPDEreg}. Therefore, in view of the previous step, we conclude that $\xi = w$. In particular, our choice of $\eta$ implies that $\Delta_b u \in H^1(B_r^+; y^b)$. This concludes the proof.
\end{proof}

\begin{cor}
\label{v-solves}
Let $u$ be a weak solution to \eqref{bil-noBC} in the sense of Definition \ref{ws-def}, and let $r<1$. Then $v \coloneqq \Delta_b u$ is a weak solution to 
\begin{equation}
\label{v-PDE}
\left\{
\arraycolsep=1.4pt\def\arraystretch{1.6}
\begin{array}{rll}
\mathcal{L}_b(v) = & 0 & \text{ in } B_r^+, \\
\partial_y^b v = & f(\cdot, u)  & \text{ on } B_r'
\end{array}
\right.
\end{equation}
in the sense that 
\[
\int_{B_r^+} y^b \nabla v \cdot \nabla \phi\,dX + \int_{B_r'} f(x, u) \phi\,dx = 0
\]
for all $\phi \in H^1_{0, S_r^+}(B_r^+; y^b) \cap L^p(B_r')$.
\end{cor}

\begin{proof}
Let $\phi \in C^{\infty}_c(B_r)$ be such that $\partial_y^b \phi = 0$ on $B_r'$. Then $\phi$ is an admissible test function according to \Cref{ws-def}, and therefore the divergence theorem (together with \Cref{LM}) implies that
\begin{equation}
\label{v-wsproof}
\int_{B_r^+} y^b \nabla v \cdot \nabla \phi\,dX + \int_{B_r'} f(x, u)\phi\,dx = 0.
\end{equation}
To conclude, we need to show that \eqref{v-wsproof} continues to hold without the additional condition $\partial_y^b \phi = 0$ on $B_r'$. To this end, let $\phi \in C^{\infty}_c(B_r)$ and define
\[
\phi_{\e}(X) \coloneqq \phi(X) - y \eta\left(\frac{y}{\e}\right) \partial_y\phi (X),
\]
where $\eta \colon \RR \to [0, 1]$ is a smooth cut-off function that is compactly supported in $(-2, 2)$ and such that $\eta \equiv 1$ on $(-1,1)$. Then $\{\phi_{\e}\}_{\e} \subset C^{\infty}_c(B_r^+)$ for $\e>0$ sufficiently small, and direct computations show that $\phi_{\e} \to \phi$ in $H^1(B_r^+; y^b)$ as $\e\to 0^+$. Moreover, as one can readily check, on $B_r'$ we have that $\phi_{\e} = \phi$ and $\partial_y^b \phi_{\e} = 0$. Consequently, by \eqref{v-wsproof} we obtain that 
\[
\int_{B_r^+} y^b \nabla v \cdot \nabla \phi_{\e}\,dX + \int_{B_r'} f(x, u)\phi_{\e}\,dx = 0.
\]
Finally,  letting $\e \to 0^+$ yields the desired result.
\end{proof}

Next, we present a proposition that establishes the $L^q$-regularity theory for both $u$ and $\Delta_b u$. This is the first step in our bootstrap process, and will enable us to apply the results of Subsection \ref{reg-subsec}. To proceed, we introduce the following structural assumption. 

\begin{enumerate}[label=(H.\arabic*), ref=H.\arabic*]\addtocounter{enumi}{3}
\item \label{H4} Let $p$ be as in \eqref{H2}. Assume that 
\[
1 \le p < p^{\#} \coloneqq \frac{2^*_b \cdot 2^{**}_b }{4} + 1,
\]
where $2^*_b$ and $2^{**}_b$ denote the fractional Sobolev exponents defined in \Cref{SII}. 
\end{enumerate}

\begin{rmk}
\label{p-condition}
Before proceeding further, we comment on the assumption \eqref{H4}. Similar constraints are standard in the literature on regularity theory for comparable problems (see, for example, Corollary 2.10 in \cite{MR3226738}); these conditions are generally dictated by integrability requirements on the specific test functions used in the proofs. In our case, as will become evident from the proof of \Cref{Lq-reg} below, assumption \eqref{H4} allows us to select a number $\beta_0$ such that
\[
2 < \beta_0 \le \gamma_0 \coloneqq \frac{2^{*}_b \cdot 2^{**}_b }{2(p - 1)}, \qquad p > 1. 
\] 
The upper bound $\beta_0 \le \gamma_0$ allows us to obtain uniform estimates with respect to the truncation we consider (see Step 2 in the proof of \Cref{Lq-reg}). The lower bound $\beta_0 > 2$ is crucial to our ability to iterate the argument (see Step 3 in the proof of \Cref{Lq-reg}). It is also worth noting that values of $p$ in $[1, 2]$ are permissible for all choices of $n$ and $b$. Additionally, we remark that when $b = 0$, the condition simplifies to
\[
p < \frac{n(n + 1)}{(n - 1)^2} + 1.
\]
In particular, for $b = 0$, the range $p \in [1, 3]$ is admissible for all physically relevant dimensions up to and including $n = 4$. In view of \eqref{H2}, this implies that our regularity results apply to a range of exponents $p$ that correspond to both sublinear and superlinear growth conditions for $f$.
\end{rmk}

\begin{prop}
\label{Lq-reg}
Let $u$ be a weak solution to \eqref{bil-noBC} in the sense of \Cref{ws-def} and let $v \coloneqq \Delta_b u$. Assume that $f$ satisfies \eqref{H2} and \eqref{H4}. Then, for every $q \in [1, \infty)$ and every $0 < r < R < 1$ there exists a constant $C = C(n, b, q, R, r)$ such that 
\begin{align}
\|u\|_{L^q(B_r^+; y^b)} + \|u\|_{L^q(B_r')} & \le C \left(\|u\|_{H^1(B_R^+; y^b)} + \|v\|_{H^1(B_R^+; y^b)}\right), \label{Lq-u} \\
\|v\|_{L^q(B_r^+; y^b)} + \|v\|_{L^q(B_r')} & \le C \left(\|u\|_{H^1(B_R^+; y^b)} + \|v\|_{H^1(B_R^+; y^b)}\right). \label{Lq-v}
\end{align}
\end{prop}

\begin{proof}
The proof adapts the approach of Proposition 7.6 in \cite{MR4169657}, which itself draws inspiration from Theorem 2.3 in \cite{MR0539217}. Our presentation highlights the key differences introduced by the nonlinearity $f$. 

Observe that, in view of \Cref{v-solves}, for every $R < 1$, the pair $(u,v)$ is a weak solution for the coupled system 
\begin{equation}
\label{x}
\left\{
\arraycolsep=1.4pt\def\arraystretch{1.6}
\begin{array}{rll}
\mathcal{L}_b(u) = & y^b v & \text{ in } B_R^+, \\
\mathcal{L}_b(v) = & 0 & \text{ in } B_R^+, \\
\partial_y^b u = & 0 & \text{ on } B_R', \\
\partial_y^b v = & f(\cdot, u)  & \text{ on } B_R'.
\end{array}
\right.
\end{equation}
We also notice that if $p = 1$ then by \eqref{H2} we have that $\partial_y^b v \in L^{\infty}(B_R')$. Therefore, in this case one can directly apply the results of Subsection \ref{reg-subsec} (see \Cref{Reg-cor} below). For this reason, in the following we assume without loss that $p > 1$. 

The proof is divided into several steps. 
\newline
\textbf{Step 1:} (\emph{Improving the integrability of $u$}) Let $\eta_{0, u} \in C^{\infty}_c(\overline{B_R^+} \setminus \overline{S_R^+}; [0, 1])$ be a cut-off function such that $\eta_{0, u} \equiv 1$ in $\overline{B_{r_0}^+}$ for some $r_0 \in (r, R)$. For $M \in \NN$, let $u_M \coloneqq \min\{|u|, M\}$ and define
\begin{equation}
\label{phi0u}
\phi_{0, u, M}(X) \coloneqq \eta_{0, u}^2(X) u_M(X)^{2^{**}_b - 2}u(X).
\end{equation}
Observe that $\phi_{0, u, M}$ is an admissible test function for the Neumann-type equation satisfied by $u$ (see \Cref{ws2-def}). In particular, testing with $\phi_{0, u, M}$ yields the identity
\begin{equation}
\label{weak-phiu}
\int_{B_R^+} y^b \nabla u \cdot \nabla \left(\eta_{0, u}^2 u_M^{2^{**}_b - 2}u\right)\,dX = - \int_{B_R^+}y^b v \left( \eta_{0, u}^2 u_M^{2^{**}_b - 2}u \right)\,dX. 
\end{equation}
Direct computations (see also Lemma 9.1 in \cite{MR2735078}) show that there exist constants $C_1 = C_1(2^{**}_b)$ and  $C_2 = C_2(2^{**}_b)$ such that
\begin{multline}
\label{phiu-young}
\int_{B_R^+} y^b \left|\nabla \left(\eta_{0, u} u_M^{(2^{**}_b - 2)/2}u\right)\right|^2\,dX \le  C_1\underbrace{\int_{B_R^+} y^b \nabla u \cdot \nabla \left(\eta_{0, u}^2 u_M^{2^{**}_b - 2}u\right)\,dX}_{I} \\ + C_2  \underbrace{\int_{B_R^+} y^b |\nabla \eta_{0, u}|^2 u_M^{2^{**}_b - 2}u^2\,dX}_{II}.
\end{multline}
In view of \eqref{weak-phiu}, using Young's inequality and the fact that $u_M \le |u|$, we obtain that
\begin{align}
| I | & \le \frac{1}{2^{**}_b} \int_{B_R^+} y^b \eta_{0, u}^2 |v|^{2^{**}_b}\,dX + \frac{2^{**}_b - 1}{2^{**}_b} \int_{B_R^+} y^b \eta_{0, u}^2 u_M^{(2^{**}_b - 2)\frac{2^{**}_b}{2^{**}_b - 1}}|u|^{\frac{2^{**}_b}{2^{**}_b - 1}}\,dX \notag \\
& \le \frac{1}{2^{**}_b} \int_{B_R^+} y^b \eta_{0, u}^2 |v|^{2^{**}_b}\,dX + \frac{2^{**}_b - 1}{2^{**}_b} \int_{B_R^+} y^b \eta_{0, u}^2 |u|^{2^{**}_b}\,dX. \label{phiu-I}
\end{align}
Moreover, we have that 
\begin{equation}
\label{phiu-II}
| II | \le \int_{B_R^+} y^b |\nabla \eta_{0, u}|^2 |u|^{2^{**}_b}\,dX. 
\end{equation}
Combining \eqref{phiu-young} with the estimates in \eqref{phiu-I} and \eqref{phiu-II} shows that 
\begin{equation}
\label{MTu}
\int_{B_R^+} y^b \left|\nabla \left(\eta_{0, u} u_M^{(2^{**}_b - 2)/2}u\right)\right|^2\,dX \le C(n, b, \eta_{0, u}) \int_{B_R^+} y^b(|u|^{2^{**}_b} + |v|^{2^{**}_b})\,dX.
\end{equation}
Since $u, v \in L^{2^{**}_b}(B_R^+; y^b)$ by \eqref{SE}, letting $M \to \infty$ in \eqref{MTu} and using Fatou's lemma, we obtain that $\nabla (\eta_{0, u} |u|^{(2^{**}_b - 2)/2}u) \in L^2(B_R^+; y^b)$. Consequently, \Cref{PI} implies that $\eta_{0, u} |u|^{(2^{**}_b - 2)/2}u \in H^1(B_R^+; y^b)$ and a further application of \eqref{S'E} and \eqref{SE} yields that $\eta_{0, u} |u|^{(2^{**}_b - 2)/2}u \in L^{2^{**}_b}(B_R^+; y^b)$ and $\eta_{0, u} |u|^{(2^{**}_b - 2)/2}u \in L^{2^*_b}(B_R')$. Recalling that $\eta_{0, u} \equiv 1$ in $\overline{B_{r_0}^+}$ we can then conclude that 
\begin{equation}
\label{reg-u-0}
u \in L^{(2^{**}_b)^2/2}(B_{r_0}^+; y^b) \cap L^{2^{*}_b \cdot 2^{**}_b /2}(B_{r_0}'),
\end{equation}
and furthermore that 
\[
\|u\|_{L^{(2^{**}_b)^2/2}(B_{r_0}^+; y^b)} + \|u\|_{L^{2^{*}_b \cdot 2^{**}_b/2}(B_{r_0}')} \le C(n, b, R, r_0) \left(\|u\|_{H^1(B_R^+; y^b)} + \|v\|_{H^1(B_R^+; y^b)}\right),
\]
where the constant $C$ depends on $R$ and $r_0$ through $\eta_{0, u}$ and a suitable rescaling of the constants $S$ and $S'$ in \Cref{SII}.
\newline
\textbf{Step 2:} (\emph{Improving the integrability of $v$}) Let $r_0 \in (r, R)$ be given as in the previous step and let $\eta_{0, v} \in C^{\infty}_c(\overline{B_{r_0}^+} \setminus \overline{S_{r_0}^+}; [0, 1])$ be a cut-off function such that $\eta_{0, v} \equiv 1$ in $\overline{B_{\rho_0}^+}$ for some $\rho_0 \in (r, r_0)$. Observe that in view of \eqref{H4}, we can find $\e > 0$ such that 
\begin{equation}
\label{H4-equiv}
p \le \frac{2^{*}_b \cdot 2^{**}_b}{2(2 + \e)} + 1.
\end{equation}
Similarly to the previous step, for $M \in \NN$ we let $v_M \coloneqq \min\{|v|, M\}$ and define
\begin{equation}
\label{phi0u}
\phi_{0, v, M}(X) \coloneqq \eta_{0, v}^2(X) v_M(X)^{\beta_0 - 2}v(X),
\end{equation}
where
\begin{equation}
\label{beta0-def}
\beta_0 \coloneqq \min\{2^*_b, 2^{**}_b, 2 + \e\},
\end{equation}
where $\e$ is given as in \eqref{H4-equiv}. Next, we claim that $\phi_{0, v, M}$ is an admissible test function for the Neumann problem satisfied by $v$ (see \eqref{v-PDE}). Since by assumption $p > 1$,  if we set
\begin{equation}
\label{gamma0-def}
\gamma_0 \coloneqq \frac{2^{*}_b \cdot 2^{**}_b }{2(p - 1)},
\end{equation}
then by \eqref{H2} and \eqref{reg-u-0} we have  $f(\cdot, u(\cdot)) \in L^{\gamma_0}(B_{r_0}')$ with
\begin{equation}
\label{f-int-est}
\|f(\cdot, u)\|_{L^{\gamma_0}(B_{r_0}')}^{\gamma_0} \le C\left(1 + \|u\|_{L^{2^{*}_b \cdot 2^{**}_b /2}(B_{r_0}')}^{2^{*}_b \cdot 2^{**}_b /2}\right).
\end{equation}
Therefore, to prove the claim, it is enough to verify that $v \in L^{\gamma_0'}(B_{r_0}')$ (here we use $'$ to denote the H\"{o}lder conjugate). By \eqref{S'E}, the desired condition follows provided that $\gamma_0' \le 2^*_b$, which in turn is equivalent to 
\[
p \le 2^{**}_b \left(\frac{2^*_b - 1}{2}\right) + 1.
\]
Observe that this condition holds true for all values of $p$ as in \eqref{H4} and therefore does not pose any further restriction. 
Arguing as in \eqref{weak-phiu} and \eqref{phiu-young}, we obtain that 
\begin{multline}
\label{reg-v-00}
\int_{B_{r_0}^+} y^b \left|\nabla \left(\eta_{0, v} v_M^{(\beta_0 - 2)/2}v\right)\right|^2\,dX \le - C_1 \underbrace{\int_{B_{r_0}'} \eta_{0, v}^2 f(x, u) v_M^{\beta_0 - 2}v\,dx}_{III} \\ + C_2 \underbrace{\int_{B_{r_0}^+} y^b |\nabla \eta_{0, v}|^2 v_M^{\beta_0 - 2}v^2\,dX}_{IV},
\end{multline}
where $C_1 = C_1(\beta_0)$ and $C_2 = C_2(\beta_0)$ are constants. Similarly to above, we readily obtain that 
\begin{align}
| III | & \le \frac{1}{\beta_0} \int_{B_{r_0}'} \eta_{0, v}^2 |f(x, u)|^{\beta_0}\,dx + \frac{\beta_0 - 1}{\beta_0}\int_{B_{r_0}'} \eta_{0, v}^2 |v|^{\beta_0}\,dx , \label{III-reg}\\
| IV | & \le \int_{B_{r_0}^+} y^b |v|^{\beta_0}|\nabla \eta_{0, v}|^2\,dX. \label{IV-reg}
\end{align}
Crucially, by \eqref{H4-equiv} and \eqref{beta0-def}, one has $\beta_0 \le 2 + \e \le \gamma_0$. Using this fact and combining the estimates in \eqref{f-int-est}, \eqref{reg-v-00}, \eqref{III-reg}, and \eqref{IV-reg} we can argue as in \eqref{reg-u-0} and conclude that $\eta_{0, v} |v|^{(\beta_0 - 2)/2}v \in H^1(B_{r_0}^+; y^b)$. Moreover, recalling that $\eta_{0, v} \equiv 1$ in $\overline{B_{\rho_0}^+}$, we obtain that
\begin{equation}
\label{reg-v-0}
v \in L^{\beta_0 2^{**}_b/2}(B_{\rho_0}^+; y^b) \cap L^{\beta_0 2^{*}_b/2}(B_{\rho_0}'),
\end{equation}
with 
\[
\|v\|_{L^{\beta_0 2^{**}_b/2}(B_r^+; y^b)} + \|v\|_{L^{\beta_0 2^{*}_b/2}(B_r')} \le C(n, b, R, r_0, \rho_0) \left(\|u\|_{H^1(B_R^+; y^b)} + \|v\|_{H^1(B_R^+; y^b)}\right),
\]
where the constant $C$ depends on $R$, $r_0$, and $\rho_0$ through $\eta_{0, u}$, $\eta_{0, v}$, and a suitable modification of the constants $S$ and $S'$ in \Cref{SII}.
\newline
\textbf{Step 3:} (\emph{The iterative argument}) The strategy employed above can be used to iteratively improve the integrability of $u$ and $v$. To this end, set $\alpha_0 \coloneqq 2^{**}_b$, let $\beta_0$ and $\gamma_0$ be given as in \eqref{beta0-def} and \eqref{gamma0-def}, respectively, and for $k \ge 1$ define
\[
\alpha_k \coloneqq \frac{2^{**}_b \beta_{k - 1}}{2}, \qquad \beta_k \coloneqq \beta_0 \left(\frac{\beta_0}{2}\right)^k, \qquad \text{ and } \qquad \gamma_k \coloneqq \frac{2^*_b \alpha_k}{2}.
\]
Observe that $\alpha_k \ge \beta_k$ for all $k$ and that $\beta_k \to \infty$ as $k \to \infty$, since $\beta_0 > 2$. For $q$ as in the statement, let $K \in \NN$ be such that $\beta_K \ge q$, set $\rho_{-1} \coloneqq R$, and for $k \ge 0$ define
\[
\rho_k \coloneqq R - \frac{k + 1}{K + 1}(R - r) \qquad \text{ and } \qquad r_k \coloneqq \frac{1}{2}\left(\rho_{k - 1} - \rho_k\right).
\]
We can then define cut-off functions $\eta_{k, u} \in C^{\infty}_c(\overline{B_{\rho_{k - 1}}^+} \setminus \overline{S_{\rho_{k - 1}}^+}; [0, 1])$ and $\eta_{k, v} \in C^{\infty}_c(\overline{B_{r_k}^+} \setminus \overline{S_{r_k}^+}; [0, 1])$ such that $\eta_{k, u} \equiv 1$ in $\overline{B_{r_k}^+}$ and $\eta_{k, v} \equiv 1$ in $\overline{B_{\rho_k}^+}$. With these at hand, one can reason as in the previous two steps by considering 
\begin{align*}
\phi_{k, u, M}(X) & \coloneqq \eta_{k, u}^2(X) u_M(X)^{\alpha_k - 2}u(X), \\
\phi_{k, v, M}(X) & \coloneqq \eta_{k, v}^2(X) v_M(X)^{\beta_k - 2}v(X)
\end{align*}
as test functions for the equations satisfied by $u$ and $v$, respectively. To be precise, for $k \le K$, reasoning as in Step 1 yields 
\[
u \in L^{2^{**}_b \alpha_k/2}(B_{r_k}^+; y^b) \cap L^{2^*_b \alpha_k/2}(B_{r_k}').
\]
In turn, similarly to Step 2, we have that $f(\cdot, u(\cdot)) \in L^{\gamma_k}(B_{r_k}')$ and therefore we obtain that
\[ 
v \in L^{2^{**}_b \beta_k/2}(B_{\rho_k}^+; y^b) \cap L^{2^*_b \beta_k/2}(B_{\rho_k}').
\]
Since \eqref{Lq-u} and \eqref{Lq-v} readily follow, this concludes the proof.
\end{proof}

\begin{cor}
\label{Reg-cor}
Let $r < 1$. Then, under the assumptions of \Cref{Lq-reg}, we have that: 
\begin{itemize}
\item[$(i)$] If $b \ge 0$ then $(u,v) \in C^{2, \alpha}(B_r^+\cup B'_r) \times C^{0, \alpha}(B_r^+\cup B'_r)$ for all $\alpha \in (0, 1 - b)$.
\item[$(ii)$] If $b < 0$ then $(u,v) \in C^{3, \alpha}(B_r^+\cup B'_r) \times C^{1, \alpha}(B_r^+\cup B'_r)$ for all $\alpha \in (0, - b)$.
\end{itemize}
\end{cor}
\begin{proof}
Fix $r < r_2 <  r_1 < 1$. By \Cref{Lq-reg} (see in particular \eqref{Lq-u}), together with \eqref{H2}, we obtain that $f(\cdot, u(\cdot)) \in L^{q}(B_{r_1}')$ for $q \in (n/(1 - b), n/(1 - b) + 1]$ if $b \ge 0$ and $q \in (-n/b, -n/b + 1]$ if $b < 0$. Thus, we are in a position to apply \Cref{STV1.5} and conclude that $v \in C^{0, \alpha}(B_{r_2}^+\cup B'_{r_2})$ for all $\alpha \in (0, 1 - b)$ whenever $b \ge 0$, and $v \in C^{1, \alpha}(B_{r_2}^+\cup B'_{r_2})$ for all $\alpha \in (0, -b)$ if $b < 0$. In turn, by \Cref{STV}, we obtain that $u \in C^{2, \alpha}(B_r^+\cup B'_{r})$ with $\alpha \in (0, 1 - b)$ for $b \ge 0$, and $u \in C^{3, \alpha}(B_r^+\cup B'_{r})$ with $\alpha \in (0, -b)$ when $b < 0$.
\end{proof}

The next proposition shows that additional regularity assumptions on $f$ yield an improvement on the regularity of $v$. To be precise, we introduce the following structural assumption.

\begin{enumerate}[label=(H.\arabic*), ref=H.\arabic*]\addtocounter{enumi}{4}
\item \label{H5} $f \colon B_1' \times \RR \to \RR$ is locally Lipschitz-continuous in the sense that for all $M > 0$ there exists a constant $C$ such that
\[
|f(x_1, \zeta_1) - f(x_2, \zeta_2)| \le C (|x_1 - x_2| + |\zeta_1 - \zeta_2|)
\]
hold for all $x_i \in B_1'$ and all $\zeta_i \in \RR$ with $|\zeta_i| \le M$, $i = 1, 2$.
\end{enumerate}

\begin{prop}
\label{improved-reg}
Let $u$ be a weak solution to \eqref{bil-noBC} in the sense of \Cref{ws-def} and let $v \coloneqq \Delta_b u$. Assume that $f$ satisfies \eqref{H2}, \eqref{H4}, and \eqref{H5}.
Then the following hold:
\begin{itemize}
\item[$(i)$] Assume that $b \ge 0$. Then for all $0 < r < R < 1$, all $\alpha \in (0, 1 - b)$, and all $q > n/(1 - b - \alpha)$ there exists a constant $C = C(n, b, R, r, q, \alpha)$ such that 
\[
\| \nabla_x' v\|_{C^{0, \alpha}({B_r^+\cup B'_r})} \le C \left(\|v\|_{H^1(B_{R}^+; y^b)} + 1 + \|\nabla_x' u\|_{L^q(B_{R}')} \right).
\]
\item[$(ii)$] Assume that $b < 0$. Then for all $0 < r < R < 1$, all $\alpha \in (0, -b)$, and all $q > n/(- b - \alpha)$ there exists a constant $C = C(n, b, R, r, q, \alpha)$ such that 
\[
\| \nabla_x' v\|_{C^{1, \alpha}(B_r^+\cup B'_r)} \le C \left(\|v\|_{H^1(B_{R}^+; y^b)} + 1 + \|\nabla_x' u\|_{L^q(B_{R}')} \right).
\]
\end{itemize}
Moreover, there exists $\beta \in (0, 1)$ such that $\partial_y^b v \in C^{0, \beta}{(B_r^+\cup B'_r)}$ regardless of the sign of $b$.
\end{prop}
\begin{proof}
Let $r_1 \coloneqq (r + R)/2$ and observe that if for $\xi \in \RR^n \setminus \{0\}$ with $|\xi|$ sufficiently small we define the function
\[
v^{\xi}(X) \coloneqq \frac{v(x + \xi, y) - v(X)}{|\xi|},
\]
then $v^{\xi}$ is a weak solution to 
\[
\left\{
\arraycolsep=1.4pt\def\arraystretch{1.6}
\begin{array}{rll}
\mathcal{L}_b(v^{\xi}) = & 0 & \text{ in } B_{r_1}^+, \\
\partial_y^b v^{\xi} = & f^{\xi}  & \text{ on } B_{r_1}',
\end{array}
\right.
\]
where 
\[
f^{\xi}(x) \coloneqq \frac{f(x + \xi, u(x + \xi, 0)) - f(x, u(x, 0))}{|\xi|}.
\]
Notice also that, by \eqref{H5}, there exists a constant $C$ (potentially depending on $\|u\|_{L^{\infty}(B_{R})}$, but independent of $\xi$) such that 
\[
|f^{\xi}(x)| \le C \left(1 + \frac{|u(x + \xi, 0) - u(x, 0)|}{|\xi|}\right).
\]
Consequently, for all $q \in [1, \infty)$ we have that
\begin{equation}
\label{fxi-uniform}
\| f^{\xi} \|_{L^q(B_{r_1}')}^q \le C\left(1 +  \int_{B_{r_1}'} \left|\frac{u(x + \xi, 0) - u(x, 0)}{|\xi|}\right|^q\,dx \right) \le C\left( 1 + \|\nabla_x' u\|_{L^q(B_{R}')}^q\right).
\end{equation}
In particular, we are in a position to apply \Cref{STV1.5} and conclude that if $b \ge 0$ then $v^{\xi} \in C^{0, \alpha}({B_r^+\cup B'_r})$ with
\begin{align}
\|v^{\xi}\|_{C^{0, \alpha}({B_r^+\cup B'_r})} & \le C \left(\|v^{\xi}\|_{L^2(B_{r_1}^+; y^b)} + \|f^{\xi}\|_{L^q(B_{r_1}')} \right) \notag \\
& \le C \left(\|v\|_{H^1(B_{R}^+; y^b)} + 1 + \|\nabla_x' u\|_{L^q(B_{R}')} \right) \label{vxi-uniform}
\end{align}
for all $\alpha \in (0, 1 - b)$, provided that $q > n/(1 - b - \alpha)$; we note that in the last step we have used \eqref{fxi-uniform} and a classical estimate for the difference quotient of Sobolev functions. Observe that the inequality obtained in \eqref{vxi-uniform} is uniform with respect to $\xi$. Therefore, the Ascoli--Arzel\'a Theorem can then be used to show that $\nabla_x' v \in C^{0, \alpha}({B_r^+\cup B'_r})$ for all $\alpha \in (0, 1 - b)$. A similar argument yields the desired result when $b < 0$.

Finally, since $x \mapsto f(x, u(x, 0))$ is a bounded Lipschitz function, then it belongs to $C^{0, \sigma}(B_R')$ for all $\sigma \in (0, 1)$. In particular, \Cref{4.5CS} implies that $\partial_y^b v \in C^{0, \beta}({B_r^+\cup B'_r})$ for some $\beta \in (0, 1)$. This concludes the proof.
\end{proof}

When $b = 0$, we can derive a slight improvement of \Cref{Reg-cor} by applying \Cref{improved-reg}.
\begin{cor}
\label{Reg-cor-0}
Under the assumptions of \Cref{improved-reg}, assume that $b = 0$ and let $r < 1$. Then, there exists $\beta \in (0,1)$ such that ${(u,v) \in C^{3, \beta}(B_r^+\cup B'_r) \times C^{1, \beta}(B_r^+\cup B'_r)}$.
\end{cor}
\begin{proof}
Fix $r < R < 1$ and observe that by \Cref{improved-reg}, when $b = 0$, we have that $v \in C^{1, \beta}({B_R^+\cup B'_R})$ for some $\beta \in (0,1)$. The desired conclusion then follows from classical regularity theory (see also \Cref{STV} $(i)$).
\end{proof}

As one can readily check, the proof of \Cref{main-reg} follows by combining the results of \Cref{Reg-cor}, \Cref{improved-reg}, and \Cref{Reg-cor-0}.

\begin{lem}
\label{vH2}
Let $u$ be a weak solution to \eqref{bil-noBC} in the sense of \Cref{ws-def} and let $v \coloneqq \Delta_b u$. Assume that $f$ satisfies \eqref{H2}, \eqref{H4}, \eqref{H5}, and furthermore that $\nabla_x' f$ and $\partial_{\zeta}f$ are Carath\'eodory functions. Then we have 
\[
v \in H^2_x(B_r^+; y^b) \qquad \text{ and } \qquad \partial_y^b v \in H^1(B_r^+; y^{-b})
\]
for all $r \in (0, R)$, $R < 1$. 
\end{lem}
\begin{proof}
In view of \Cref{FS-reg}, it is enough to show that $x \mapsto \gamma_2(x) \coloneqq f(x, u(x, 0))$ belongs to  $W^{1, \frac{2n}{n + 1 - b}}(B_R')$. Since
\[
\nabla_x' \gamma_2(x) = \nabla_x'f(x, u(x, 0)) + \partial_{\zeta}f(x, u(x, 0))\nabla_x'u(x, 0),
\]
the desired result readily follows from the regularity of $u$ (see \Cref{Reg-cor}) and \eqref{H5}.
\end{proof}

\begin{rmk}
\label{reg-v-rmk}
We conclude the section with some comments on the regularity of solutions and the techniques used. 

First, we recall that the standard pathway to proving H\"older regularity for solutions to \emph{stable} obstacle problems for the Laplacian (or its fractional counterparts) revolves around Morrey's Dirichlet growth theorem (see, for example, Theorem 3.5.2 in \cite{MR2492985}). A key step in obtaining the gradient estimate required to apply Morrey's theorem is a Caccioppoli inequality, typically derived by either showing that $u^{\pm}$ are subharmonic (or $b$-subharmonic, respectively), or by directly leveraging the fact that the quantities $f(u)u^+$ and $f(u)u^-$ have the correct sign (hence the distinction between stable and unstable cases). For more details on how these strategies have been successfully employed in previous works, we refer the reader to Section 3 of \cite{MR2995409}, Sections 4 and 5 of \cite{MR3348118}, Section 3 of \cite{MR4144102}, and Section 2 of \cite{DO23}.

In the present case, this approach does not seem to be viable due to the different structure of our fourth-order problem. Indeed, using any suitable modification of $v$ as a test function leads to a boundary term involving $f(\cdot, u)v$, which has no known sign. Similar difficulties arise when approaching the problem from an energy perspective, noting that $v$ is a minimizer of
\[
\mathcal{K}(w) \coloneqq \int_{B_r^+} y^b|\nabla w|^2\,dX + 2 \int_{B_r'} f(x, u)w\,dx
\]
over the class 
\[
\mathcal{A}(r, v) \coloneqq \{ w \in H^1(B_r^+; y^b) : w = v \text{ on } S_r^+ \}.
\]

We also mention that while we relied on a similar concept of stability in the form of \eqref{H3} for our existence result (see \Cref{MIN}), the assumption is not required for any of the results in this section.

Finally, it is worth noting that not only the approach described above cannot be directly applied, but subharmonicity may generally fail, as shown by the following one-dimensional example: let $b = 0$, $p = 2$, $\lambda_- = \lambda_+ = 1$, and set $u(x) \coloneqq -x^3 + 6$. Then $u$ satisfies 
\[
\left\{
\arraycolsep=1.4pt\def\arraystretch{1.6}
\begin{array}{rl}
u^{(iv)} = & 0 \qquad \mbox{on }(0,1), \\
u'''(0) = & -6 = \lambda_- u^-(0) - \lambda_+u^+(0), \\
u'(0) = & 0.
\end{array}
\right.
\]
Moreover $u > 0$ on $(0,1)$, but $u''(x) = -6x < 0$, thus showing that $u$ is superharmonic.
\end{rmk}

\section{Almgren's frequency formula}
\label{AFF-sec}
Almgren's frequency functions, defined as the local ratio of energy over mass, have been extensively used to study the asymptotic behavior of solutions to a wide range of problems. In the following we will consider the \emph{classical} frequency function, defined via 
\begin{equation}
\label{N0-def}
N_0(r, u, v) \coloneqq \frac{rD_0(r, u, v)}{H(r, u, v)},
\end{equation}
where 
\begin{equation}
\label{D0-def}
D_0(r, u, v) \coloneqq \int_{B_r^+} y^b(|\nabla u|^2 + |\nabla v|^2)\,dX
\end{equation}
and 
\begin{equation}
\label{H-def}
H(r, u, v) \coloneqq \int_{S_r^+} y^b(u^2 + v^2)\,d\Hh^n,
\end{equation}
as well as the \emph{perturbed} frequency function
\[
N(r, u, v) \coloneqq \frac{r D(r, u, v)}{H(r, u, v)},
\]
where
\begin{equation}
\label{D-def}
D(r, u, v) \coloneqq \int_{B_r^+} y^b(|\nabla u|^2 + |\nabla v|^2 + u \Delta_b u)\,dX + \int_{B_r'} (\partial_y^b v) v\,dx.
\end{equation}
The latter was introduced in \cite{FF20}, where the authors carried out a classification of blow-up profiles and used it to prove a unique continuation result for solutions to the fourth order problem $\Delta u = v$, $\Delta v = 0$, with $\partial_y u = 0$ and Neumann boundary coupling given by $\partial_y v = h u$ on the thin space. Here $h$ denotes a given function in $C^1(B_R')$.

As explained in \cite{FF20}, the proof of the monotonicity argument for $N$ is rather delicate due to the presence of the term 
\begin{equation}
\label{bob}
\int_{\partial B_r'} h u v\,d\Hh^{n - 1}
\end{equation}
in the expression for $\tfrac{d}{dr}D(r, u, v)$. In \cite{FF20}, estimates for this ``boundary of the boundary'' term are obtained via careful applications of the divergence theorem which ultimately allow one to recast it in terms of bulk and boundary terms that can be controlled (see Lemma 2.12 in \cite{FF20}). 

Although the techniques developed in \cite{FF20} for the case $f(\cdot, u) = hu$ suggest potential generalizations for instances where $h$ is not necessarily differentiable, the lack of differentiability in the second variable for the prototypical nonlinearities considered in our study of obstacle-type problems necessitates a different approach. For this reason, in the following we obtain a quasi-monotonicity result for the Almgren functional $N_0$ (see \Cref{AM-thm} below). It is worth noting that our proof continues to involve the modified energy $D(r, u, v)$, as it naturally arises in the expression for $H'(r, u, v)$ (see \Cref{D_0H-der} below). However, by considering the classical Dirichlet energy $D_0(r, u, v)$ we are able to obtain the desired estimates without ever having to consider derivatives of $D(r, u, v)$, thus avoiding some of the difficulties described above. 

\begin{rmk}
\label{DHA-comparison}
Our work, although related to \cite{MR4316741}, addresses different aspects beyond just the extension to more general elliptic operators. Here a word of caution is necessary:
\begin{itemize}
\item[$(i)$] The nonlinearity considered in \cite{MR4316741} is $f(u) \coloneqq \lambda^- (u^-)^{p - 1} - \lambda^+ (u^+)^{p - 1}$. The lack of differentiability mentioned earlier forced the authors to restrict their attention to the case $p \ge 3$. Specifically, in \cite{MR4316741} the problematic term in \eqref{bob} is estimated in terms of 
\[
\nabla f'(u) = -(p - 1)(p - 2)(- \lambda_{-}(u^-)^{p - 3} + \lambda_{+}(u^+)^{p - 3})\nabla u.
\]
Thus, when $p \ge 3$, $\nabla f'(u)$ is well-defined and can be estimated in terms of norms of the solution. Notably, the proof of Theorem 4.1 in \cite{MR4316741} requires $L^{\infty}$ bounds on $u$ (see also the proof of \Cref{P-lem} below). As mentioned in \Cref{reg-v-rmk}, it is unclear how to verify this assumption for all $p \ge 3$. 
\item[$(ii)$] Furthermore, in \cite{MR4316741} the authors obtain results for $p = 2$. These, however, rely on the additional assumption that $v = 0$ on $\{ u = 0 \}$ (see Remark 2.9 in \cite{MR4316741}) as a way to prevent problematic terms involving the integral $v$ over the set $\partial\{ u > 0\} \cap B_r^+$ from playing a role in the analysis (see, for example, Equation (4.9) in \cite{MR4316741}). To the best of our understanding, this assumption cannot be verified in general (see \Cref{reg-v-rmk}). Observe that this difficulty was not present in \cite{FF20}, where in view of the smoothness of $f$, applications of the divergence theorem do not yield such terms. 
\end{itemize} 
For these reasons, the techniques developed in this section (and more in general, the main results of this paper) focus on the case $p \in [2, p^\#)$ and differ from some of the arguments in \cite{MR4316741} even when $b = 0$ and $p = 2$ or $p \ge 3$. We recall here that $p^{\#}$ is the critical exponent defined in \eqref{H4}.
\end{rmk}

Next, we define 
\begin{equation}
\label{P-def}
P(r, u, v) \coloneqq \int_{B_r^+} y^b u \Delta_b u\,dX + \int_{B_r'} (\partial_y^b v) v\,dx
\end{equation}
so that $D(r, u, v) = D_0(r, u, v) + P(r, u, v)$. 

Throughout the rest, we assume that $u$ is a weak solution to \eqref{bil-noBC} in the sense of \Cref{ws-def}. To simplify the notation, we write $N_0(r)$ for $N_0(r, u, \Delta_b u)$ and use similar conventions for all other functionals defined above. In particular, we have that
\[
P(r) = \int_{B_r^+} y^b u v\,dX + \int_{B_r'} f(x, u) v\,dx.
\]

In the following we assume that $f$ satisfies more stringent versions of \eqref{H2} and \eqref{H4}.
\begin{enumerate}[label=(H.\arabic*$'$), ref=H.\arabic*$'$]\addtocounter{enumi}{1}
\item \label{H2'} Assume that $f \colon B_1' \times \RR \to \RR$ is Carath\'eodory and satisfies 
\[
|f(x, \zeta)| \le C|\zeta|^{p - 1}
\] 
for $\mathcal{L}^{n}$-a.e.\@ $x \in B_1'$ and all $\zeta \in \RR$.
\end{enumerate} 
\begin{enumerate}[label=(H.\arabic*$'$), ref=H.\arabic*$'$]\addtocounter{enumi}{3}
\item \label{H4'} Let $p$ be as in \eqref{H2'}. Assume that 
\[
2 \le p < p^{\#} \coloneqq \frac{2^*_b \cdot 2^{**}_b }{4} + 1,
\]
where $2^*_b$ and $2^{**}_b$ denote the fractional Sobolev exponents defined in \Cref{SII}. 
\end{enumerate} 


We then have the following result.

\begin{lem}
\label{P-lem}
Let $D_0, H$, and $P$ be given as in \eqref{D0-def}, \eqref{H-def}, and \eqref{P-def}, respectively. Assume that $f$ satisfies \eqref{H2'}, \eqref{H4'}, and \eqref{H5}. Then
\[
|P(r)| \le Cr^{-b}(r D_0(r) + H(r)),
\]
where $C$ is a constant that is independent of $r$ but that possibly depends on $\|u\|_{L^{\infty}(B_r')}$ when $p > 2$.
\end{lem}
\begin{proof}
We begin by observing that by \Cref{PI} we have that
\begin{equation}
\label{P-est1}
\left|\int_{B_r^+} y^b u v \,dX\right| \le \frac{1}{2} \int_{B_r^+} y^b(u^2 + v^2)\,dX \le Cr(r D_0(r) + H(r)).
\end{equation}
Moreover, notice that \eqref{H2'}, \eqref{H4'}, and \Cref{TI} can be used to show that 
\begin{equation}
\label{P-est2}
\left|\int_{B_r'} f(x, u)v\,dx\right| \le C \int_{B_r'}|u|^{p - 1}|v|\,dx \le Cr^{-b}(r D_0(r) + H(r)),
\end{equation}
where for $p > 2$ the constant $C$ in the last inequality depends also on $\|u\|_{L^{\infty}(B_r')}^{p - 2}$ (which, we recall, is uniformly bounded by \Cref{Reg-cor}). Combining \eqref{P-est1} and \eqref{P-est2} yields the desired inequality.
\end{proof}

\begin{lem}
\label{D-ibp}
Let $D$ be as in \eqref{D-def}. Then, for $\mathcal{L}^1$-a.e.\@ $r \in (0, 1)$ we have that 
\[
D(r) = \int_{S_r^+} y^b(uu_{\nu} + vv_{\nu})\,d\Hh^n.
\]
\end{lem}
\begin{proof}
For every $R < 1$ we have that 
\begin{equation}
\label{D-slice}
\int_{B_R^+} y^b\left|u \nabla u \cdot \frac{X}{|X|} + v \nabla v \cdot \frac{X}{|X|}\right|\,dX = \int_0^R \int_{S_r^+} y^b |u u_{\nu} + v v_{\nu}|\,d\Hh^n\,dr.
\end{equation}
Since the left-hand side of \eqref{D-slice} is bounded, if we set 
\begin{equation}
\label{d-def}
d(r) \coloneqq \int_{S_r^+} y^b(u u_{\nu} + v v_{\nu})\,d\Hh^n,
\end{equation}
we obtain that $d \in L^1(0, R)$ and $d(r)$ is therefore well-defined for $\mathcal{L}^1$-a.e.\@ $r \in (0, R)$. In turn, for every such $r$, the desired identity follows from an application of the divergence theorem. This concludes the proof.
\end{proof}

The next result shows that, without loss of generality, we can assume that $H(r) > 0$ for all $r$ sufficiently small. This, in turn, ensures that the frequency functions $N_0$ and $N$ are well-defined. 

\begin{lem}
\label{H>0}
Under the assumptions of \Cref{P-lem}, there exists $r_0 \in (0, 1)$ such that either $H(r) > 0$ for all $r < r_0$ or $u$ and $v$ are identically equal to zero in $B_{r_0}^+$.
\end{lem}
\begin{proof} We divide the proof into two steps.
\newline
\textbf{Step 1:} Let $C$ be the constant in \Cref{P-lem}. Then we can find $r_0 > 0$ such that 
\[
0 < 1 - Cr_0^{1 - b}.
\]
Arguing by contradiction, assume that $u$ and $v$ are not both identically equal to zero in $B_{r_0}^+$ and that there exists $r < r_0$ such that $H(r) = 0$. In this step, we assume that $r$ is a Lebesgue point of both $d$ and $H$; this assumption will be removed in the next step by an application of \Cref{main-reg}. Let $\delta > 0$ be such that $r + \delta < r_0$. Since by \Cref{D-ibp} we have that $D(\rho) = d(\rho)$ for $\mathcal{L}^1$-a.e.\@ $\rho \in (r, r + \delta)$ (see also \eqref{d-def}), it follows from \Cref{P-lem} that
\[
\fint_{r}^{r + \delta} D_0(\rho)\,d\rho = \fint_{r}^{r + \delta} (d(\rho) - P(\rho))\,d\rho \le \fint_{r}^{r + \delta} d(\rho)d\rho + \fint_{r}^{r + \delta} C\rho^{-b}(\rho D_0(\rho) + H(\rho))\,d\rho.
\]
Observe that rearranging the terms in the previous inequality yields 
\begin{equation}
\label{avg-d}
\fint_{r}^{r + \delta} (1 - C \rho^{1 - b})D_0(\rho)\,d\rho \le \fint_{r}^{r + \delta} d(\rho)\,d\rho + \fint_{r}^{r + \delta} C\rho^{-b}H(\rho)\,d\rho.
\end{equation}
Using the fact that $\rho \mapsto D_0(\rho)$ is non-decreasing and in view of our choice of $r_0$, by letting $\delta \to 0^+$ in \eqref{avg-d} we obtain that
\begin{align*}
(1 - C r_0^{1 - b}) D_0(r) & \le \limsup_{\delta \to 0^+} \fint_{r}^{r + \delta} (1 - C \rho^{1 - b})D_0(\rho)\,d\rho \\
& \le \lim_{\delta \to 0^+} \left(\fint_{r}^{r + \delta} d(\rho)\,d\rho + \fint_{r}^{r + \delta} C\rho^{-b}H(\rho)\,d\rho\right) = d(r) + Cr^{-b}H(r) = d(r).
\end{align*}
Observe that $H(r) = 0$ implies that $(u, v) \equiv 0$ on $S_r^+$. Then $d(r) = 0$ and therefore, recalling that $D_0$ is nonnegative, we must also have that $D_0(r) = 0$. In particular, since $u$ and $v$ vanish on $S_r^+$, we can conclude that $u, v \equiv 0$ in $B_r^+$. By classical unique continuation for elliptic operators with smooth coefficients (see, for example, Theorem 3 in \cite{MR1159832}) we obtain that $(u, v) \equiv 0$ in $B_{r_0}^+ \cap \{ y > \tau \}$ for all $\tau > 0$. Letting $\tau \to 0^+$ leads to a contradiction.
\newline
\textbf{Step 2:} To complete the proof, it is enough to show that every $r < r_0$ is a Lebesgue point for $d$ and $H$. We establish the result for $d$. The argument for $H$ is analogous and involves simpler computations, so we omit the details. To this end, let $\nu = X/|X|$ and observe that by \Cref{improved-reg}, for $b \le 0$ there exists $\alpha > 0$ such that $uu_{\nu} + vv_{\nu} \in C^{0, \alpha}(B_{r_0}^+ \cup B_{r_0}')$; in order to simplify the notation, here we use $w$ to denote this function. Note also that, by a change of variables, 
\[
d(r) = r^{n + b} \int_{S_1^+} y^b w(rX)\,d\Hh^n.
\]
In turn, for every $r_1, r_2 \in (0, r_0)$ we have that 
\begin{align*}
|d(r_1) - d(r_2)| & = \left|r_1^{n + b} \int_{S_1^+} y^b w(r_1X)\,d\Hh^n - r_2^{n + b} \int_{S_1^+} y^b w(r_2X)\,d\Hh^n\right| \\
& \le r_1^{n + b}\int_{S_1^+}y^b|w(r_1X) - w(r_2X)|\,d\Hh^n + \left|r_1^{n + b} - r_2^{n + b}\right|\int_{S_1^{+}}y^b|w(r_2X)|\,d\Hh^n \\
& \le \left(Cr_1^{n + b}|r_1 - r_2|^{\alpha} + \left|r_1^{n + b} - r_2^{n + b}\right|\|w\|_{L^{\infty}(B_{r_0}^+)}\right) \int_{S_1^+}y^b\,d\Hh^n.
\end{align*}
This readily implies that $d \in C(0, r_0)$, and the desired result follows in this case. Similarly, another application of \Cref{improved-reg} in the case $b > 0$ yields the existence of $\alpha > 0$ such that $y^b(uu_{\nu} + vv_{\nu}) \in C^{0, \alpha}(B_{r_0}^+ \cup B_{r_0}')$. The remainder of the argument follows as in the previous case. This completes the proof.
\end{proof}

\begin{lem}
\label{D_0H-der}
Under the assumptions of \Cref{vH2}, let $D_0$ and $H$ be as in \eqref{D0-def} and \eqref{H-def}, respectively. Then, $D_0, H \in W^{1, 1}_{\loc}(0, R)$ for all $R < 1$. Moreover, we have that 
\begin{multline}
\label{D0'-Rellich}
D_0'(r) = \frac{n + b - 1}{r}D_0(r) + 2 \int_{S_r^+} y^b(u_{\nu}^2 + v_{\nu}^2)\,d\Hh^n \\ - \frac{2}{r} \int_{B_r^+}y^b(X \cdot \nabla u)v\,dX - \frac{2}{r}\int_{B_r'}(x \cdot \nabla_x' v)f(x, u)\,dx
\end{multline}
and 
\begin{equation}
\label{H'formula}
H'(r) = \frac{n + b}{r}H(r) + 2D(r)
\end{equation}
hold for $\mathcal{L}^1$-a.e.\@ $r \in (0, R)$.
\end{lem}
\begin{proof}
Using the fact that $u, v \in H^1(B_R^+; y^b)$ and arguing as in the proof of \Cref{D-ibp}, we obtain that
\begin{equation}
\label{D_0'}
D_0'(r) = \int_{S_r^+} y^b(|\nabla u|^2 + |\nabla v|^2)\,d\Hh^n
\end{equation}
for $\mathcal{L}^1$-a.e.\@ $r \in (0, R)$. Also, observe that in view of \Cref{vH2}, the boundary conditions satisfied by $u$ and $v$ (see \eqref{bil}), and \eqref{D_0'}, an application of \Cref{Rellich} to both functions $u$ and $v$ implies \eqref{D0'-Rellich}. Next, as can be readily verified through direct computations (see also Lemma 3.8 in \cite{MR3169789}), for $\mathcal{L}^1$-a.e.\@ $r \in (0, R)$, we have that
\[
H'(r) = \frac{n + b}{r} H(r) + 2 \int_{S_r^+} y^b(u u_{\nu} + v v_{\nu})\,d\Hh^n = \frac{n + b}{r} H(r) + 2 D(r),
\]
where in the last step we have used \Cref{D-ibp}. This concludes the proof.
\end{proof}

In the following we assume that $f$ satisfies a slightly more stringent version of \eqref{H5}.
\begin{enumerate}[label=(H.\arabic*$'$), ref=H.\arabic*$'$]\addtocounter{enumi}{4}
\item \label{H5'} $f \colon B_1' \times \RR \to \RR$ is locally Lipschitz-continuous in the sense of \eqref{H5}. Moreover, $\nabla_x' f$ and $\partial_{\zeta} f$ are Carath\'eodory functions with the property that for every $M > 0$ there exists a constant $C > 0$ such that 
\[
|\nabla_x' f(x, \zeta)| \le C|\zeta| \qquad \text{ and } \qquad |\partial_{\zeta} f(x, \zeta)| \le C
\]
for $\mathcal{L}^n$-a.e. $x \in B_1'$ and all $\zeta \in \RR$ with $|\zeta| \le M$.
\end{enumerate} 

With these results at hand, we can now state and prove the main theorem of the section.

\begin{thm}
\label{AM-thm}
Let $u$ be a weak solution to \eqref{bil-noBC} in the sense of \Cref{ws-def}. Assume that $b \le 0$ and let $N_0$ be given as in \eqref{N0-def}. Furthermore, assume that $f$ satisfies \eqref{H2'}, \eqref{H4'}, and \eqref{H5'}. Let $r_0$ be as in \Cref{H>0}. Then there exists a constant $C > 0$ such that 
\[
r \mapsto - \frac{1}{1 + N_0(r)} + Cr
\] 
is non-decreasing in $(0, r_0)$. In particular, there exists $\mu \in [0, \infty]$ such that
\begin{equation}\label{frequency_def}
\lim_{r \to 0^+} N_0(r) = \mu.
\end{equation}
\end{thm}
\begin{proof}
We divide the proof into several steps.
\newline
\textbf{Step 1:} In this step, we obtain a preliminary lower bound on $N_0'(r)$. Observe that by \Cref{H>0}, without loss of generality we can assume that there exists $r_0 > 0$ such that $H(r) > 0$ for all $r < r_0$, since otherwise there is nothing to prove. Then, by \Cref{D_0H-der} we have that $N_0 \in W^{1,1}_{\loc}(0, r_0)$, and for $\mathcal{L}^1$-a.e.\@ $r < r_0$ we have that
\begin{equation}
\label{N_0'=}
N_0'(r) = \frac{D_0(r)}{H(r)} + \frac{r D_0'(r)}{H(r)} - \frac{r D_0(r) H'(r)}{H(r)^2}.
\end{equation}
Moreover, in view of \eqref{H'formula} and recalling that $P$ is the function defined in \eqref{P-def}, from \eqref{N_0'=} we obtain that
\begin{align}
N_0'(r) & = -(n + b - 1) \frac{D_0(r)}{H(r)} + \frac{r D_0'(r)}{H(r)} - \frac{2rD_0(r)D(r)}{H(r)^2}, \notag \\
& = -(n + b - 1) \frac{D_0(r)}{H(r)} + \frac{r D_0'(r)}{H(r)} - \frac{2rD(r)^2}{H(r)^2} + \frac{2r P(r)D(r)}{H(r)^2}. \label{N0'}
\end{align}
Using the fact that by the Cauchy--Schwarz inequality and \Cref{D-ibp}
\[
H(r)\int_{S_r^+} y^b(u_{\nu}^2 + v_{\nu}^2)\,d\Hh^n - D(r)^2 \ge 0,
\] 
substituting the expression \eqref{D0'-Rellich} for $D_0'(r)$  into \eqref{N0'} yields 
\begin{equation}
\label{N0'-lower}
N_0'(r) \ge - \frac{2}{H(r)}\int_{B_r^+}y^b(X \cdot \nabla u)v\,dX - \frac{2}{H(r)}\int_{B_r'}(x \cdot \nabla_x' v)f(x, u)\,dx + \frac{2r P(r)D(r)}{H(r)^2}
\end{equation}
holds for $\mathcal{L}^1$-a.e.\@ $r < r_0$.
\newline
\textbf{Step 2:} In this step we prove estimates for the quantities that appear on the right-hand side of \eqref{N0'-lower}. These will allow us to recast \eqref{N0'-lower} as a differential inequality that can be integrated to yield the desired result. To begin, observe that by Young's inequality and \Cref{PI} we obtain that 
\[
\left| \int_{B_r^+}y^b(X \cdot \nabla u)v\,dX \right| \le \frac{r}{2} \int_{B_r^+} y^b(|\nabla u|^2 + v^2)\,dX \le Cr (D_0(r) + rH(r)). 
\]
In particular, this implies that 
\begin{equation}
\label{rhs1}
- \frac{2}{H(r)}\int_{B_r^+}y^b(X \cdot \nabla u)v\,dX \ge - C(N_0(r) + r^2).
\end{equation}
Next, consider the field 
\[
G(X) \coloneqq (x v_y(X) f(x, u(X)), - (x \cdot \nabla_x' v(X))f(x, u(X)))
\] 
and notice that the divergence theorem applied to $G$ yields the identity
\begin{multline}
\label{rhs2}
\int_{B_r'}(x \cdot \nabla_x' v)f(x, u)\,dx = \frac{1}{r} \int_{S_r^+} (y(x \cdot \nabla_x' v) - |x|^2v_y)f(x, u)\,d\Hh^n \\
+ \int_{B_r^+} [n f(x, u)v_y + (x \cdot \nabla_x' f(x, u))v_y + (x \cdot \nabla_x' u)\partial_{\zeta}f(x, u)v_y - (x \cdot \nabla_x' v)\partial_{\zeta}f(x, u)u_y]\,dX.
\end{multline}
We proceed by estimating the terms on the right-hand side of \eqref{rhs2} separately. In the following estimates we rely on the fact that $1 \le y^b$ for $b \le 0$. To begin, observe that arguing as in \eqref{P-est2} and applying Young's inequality we obtain that
\begin{align}
\left|\frac{1}{r} \int_{S_r^+} (y(x \cdot \nabla_x' v) - |x|^2v_y)f(x, u)\,d\Hh^n \right| & \le Cr \int_{S_r^+} (|\nabla_x' v| + |v_y|)|u|^{p-1}\,d\Hh^n \notag \\
& \le Cr \int_{S_r^+} y^b\left(r|\nabla v|^2 + \frac{u^2}{r}\right)\,d\Hh^n, \label{ref_asked}
\end{align}
where for $p > 2$ the constant $C$ in the last inequality depends also on $\|u\|_{L^{\infty}(B_{r_0}')}^{p - 2}$. Consequently, in view of \eqref{D_0'}, for $\mathcal{L}^1$-a.e.\@ $r < 1$ we have that
\begin{equation}
\left|\frac{1}{r} \int_{S_r^+} (y(x \cdot \nabla_x' v) - |x|^2v_y)f(x, u)\,d\Hh^n \right| \le C(r^2D_0'(r) + H(r)). \label{G-boundary}
\end{equation}
The second integral on the right-hand side of \eqref{rhs2} can be estimated in a similar fashion. Indeed, by \eqref{H5'} we have that for every $R < 1$ there exists a constant $C > 0$, possibly depending on $\|u\|_{L^{\infty}(B_R^+)}$, such that
\[
|\nabla_x' f(x, u(X))| \le C|u(X)| \qquad \text{ and } \qquad |\partial_{\zeta} f(x, u(X))| \le C
\]
holds for $\mathcal{L}^{n + 1}$-a.e.\@ $X \in B_R^+$. Then, for all $r < R < 1$ we have that
\begin{align}
\left|\int_{B_r^+} n f(x, u)v_y \,dX \right| & \le C(rD_0(r) + H(r)), \label{divbulk1} \\ 
\left|\int_{B_r^+} (x \cdot \nabla_x' u) \partial_{\zeta}f(x, u)v_y - (x \cdot \nabla_x' v) \partial_{\zeta}f(x, u) u_y)\,dX\right| & \le C r D_0(r), \label{divbulk2} \\
\left|\int_{B_r^+} (x \cdot \nabla_x' f(x, u))v_y\,dX\right| & \le Cr(D_0(r) + r H(r)). \label{divbulk3}
\end{align}
Combining the estimate in \eqref{G-boundary} with those in \eqref{divbulk1}, \eqref{divbulk2}, and \eqref{divbulk3}, and recalling the identity in \eqref{rhs2} we have that 
\begin{equation}
\label{rhs2-estimate}
- \frac{2}{H(r)}\int_{B_r'}(x \cdot \nabla_x' v)f(x, u)\,dx \ge - C\left(\frac{r^2 D_0'(r)}{H(r)} + 1 + N_0(r)\right)
\end{equation}
holds for $\mathcal{L}^1$-a.e.\@ $r < r_0$. Furthermore, by \Cref{P-lem}, the remaining term on the right-hand side of \eqref{N0'-lower} can be bounded from below as follows: 
\begin{equation}
\label{rhs3-estimate}
\frac{2rP(r)D(r)}{H(r)^2} \ge - \frac{2r|P(r)|D_0(r)}{H(r)^2} \ge - C r^{-b}(N_0(r)^2 + N_0(r)).
\end{equation}
Finally, substituting \eqref{rhs1}, \eqref{rhs2-estimate}, and \eqref{rhs3-estimate} into \eqref{N0'-lower} shows that
\begin{equation}
\label{N0'-lower2}
N_0'(r) \ge -C\left(1 + N_0(r) + N_0(r)^2 + \frac{r^2 D_0'(r)}{H(r)} \right)
\end{equation}
holds for $\mathcal{L}^1$-a.e.\@ $r < r_0$. \newline
\textbf{Step 3:} In view of \eqref{N_0'=} and \eqref{H'formula}, for $\mathcal{L}^1$-a.e.\@ $r < r_0$ such that $N_0'(r) \le 0$, we have that
\[
\frac{r D_0'(r)}{H(r)} \le (n + b - 1)\frac{D_0(r)}{H(r)} + \frac{2r D_0(r) D(r)}{H(r)^2}.
\]
This, together with \Cref{P-lem}, implies that for every such $r$ we have
\begin{align*}
\frac{r D_0'(r)}{H(r)} & \le (n + b - 1)\frac{D_0(r)}{H(r)} + \frac{2r D_0(r)^2}{H(r)^2} + \frac{Cr^{1 - b}D_0(r)(r D_0(r) + H(r))}{H(r)^2} \notag \\
& \le (n + b - 1 + Cr^{1 - b})\frac{D_0(r)}{H(r)} + (2 + Cr^{1 - b}) \frac{rD_0(r)^2}{H(r)^2},
\end{align*}
and therefore we have that 
\begin{equation}
\label{D0'-upper}
\frac{r^2 D_0'(r)}{H(r)} \le C(N_0(r) + N_0(r)^2).
\end{equation}
Substituting \eqref{D0'-upper} into \eqref{N0'-lower2} yields
\begin{equation}
\label{N0'-lower3}
\frac{N_0'(r)}{1 + N_0(r) + N_0(r)^2} \ge -C.
\end{equation}
We remark that while \eqref{N0'-lower3} was obtained under the assumption that $N_0'(r) \le 0$, the estimate remains true also when $N_0'(r) > 0$. Hence we conclude that \eqref{N0'-lower3} holds for almost all $r \in (0, r_0)$. Next, observe that \eqref{N0'-lower3} implies 
\begin{equation}
\label{N0'-lower4}
\frac{N_0'(r)}{(1 + N_0(r))^2} \ge -C.
\end{equation}
Integrating \eqref{N0'-lower4} over $(r_1, r_2) \subset (0, r_0)$ yields 
\[
C(r_1 - r_2) \le \int_{r_1}^{r_2} \frac{N_0'(r)}{(1 + N_0(r))^2}\,dr = \frac{1}{1 + N_0(r_1)} - \frac{1}{1 + N_0(r_2)},
\]
which can be rewritten as 
\begin{equation}
\label{frac-mon}
-\frac{1}{1 + N_0(r_1)} + C r_1 \le - \frac{1}{1 + N_0(r_2)} + C r_2.
\end{equation}
Observe that \eqref{frac-mon} readily implies that $N_0$ admits a nonnegative limit as $r \to 0^+$; therefore, this concludes the proof.
\end{proof}

\begin{rmk}
Notably, nearly the entire proof of \Cref{AM-thm} remains valid for $b > 0$, with the exception of the estimates in \eqref{ref_asked}. In essence, since it does not seem possible to bound the term 
\[
\int_{B_r'} (x \cdot \nabla_x' v) \partial_y^b v\,dx
\] 
in terms of $D_0(r)$ and $H(r)$, we used our boundary conditions to replace $y^b \partial_y v$ with $f(\cdot, u)$ on $B_r'$ and proceeded by bounding the resulting integral in terms of Sobolev norms of $f(\cdot, u)$ in $B_r^+$. However, after this substitution, it is unclear whether (or how) the degenerate weight $y^b$, for $b > 0$, can be reintroduced in the estimates. Achieving this would likely require a different approach.
\end{rmk}

The rest of the analysis is facilitated by the introduction of an opportunely rescaled version of the function $H$, namely
\begin{equation}
\label{h-def}
h(r) \coloneqq \frac{H(r)}{r^{n + b}} = \frac{1}{r^{n + b}} \int_{S_r^+} y^b(u^2 + v^2)\,d\Hh^n.
\end{equation}

The following result is inspired by its counterparts in previous investigations (see, for example, Lemma 2.15 in \cite{FF20}).

\begin{lem}
\label{double-lem}
Under the assumptions of \Cref{AM-thm}, let $h$ be as in \eqref{h-def} and $\mu$ as in \eqref{frequency_def}. Assume that $\mu < \infty$. Then the following hold:
\begin{itemize}
\item[$(i)$] for every $0 < r < R < r_0$, we have that
\[
r^{-2\mu} h(r) \le R^{-2\mu} h(R) e^{C(R - r)},
\]
where $C$ is a constant that depends on $\mu$ and on $\|u\|_{L^{\infty}(B_{r_0}')}^{p - 2}$ (when $p > 2$). In particular, $h(r) = \mathcal{O}(r^{2\mu})$ as $r \to 0^+$.
\item[$(ii)$] For every $\delta > 0$ there exists $r_{\delta} > 0$ such that for every $0 < r < R < r_{\delta}$, we have
\[
R^{-2\mu} h(R) \left(\frac{r}{R}\right)^{\delta} \le r^{-2\mu} h(r).
\]
In particular, there exists a constant $C = C(\delta)$ such that $h(r) \ge C(\delta) r^{2\mu + \delta}$ for all $0 < r < r_{\delta}$.
\end{itemize}
\end{lem}
\begin{proof}
Observe that, in view of \eqref{H'formula}, for $\mathcal{L}^1$-a.e.\@ $r < 1$ we have that
\[
h'(r) = \frac{2 D(r)}{r^{n + b}}.
\]
Moreover, since for every such $r < r_0$ we have that $H(r) > 0$ (see \Cref{H>0}), direct computations show that 
\begin{equation}
\label{logh}
\frac{d}{dr} \log (r^{-2\mu}h(r)) = \frac{2}{r}\left(N_0(r) - \mu + \frac{rP(r)}{H(r)}\right).
\end{equation}
We now proceed by estimating the right-hand of \eqref{logh}, thus obtaining differential inequalities that can be integrated. To this end, observe that by \Cref{AM-thm}, for every $r \in (0, r_0)$ we have that 
\begin{equation}
\label{N0-mu0}
N_0(r) - \mu \ge -\frac{Cr(1 + \mu)^2}{1 + Cr(1 + \mu)}.
\end{equation}
Combining \eqref{logh} with \eqref{N0-mu0} and by an application of \Cref{P-lem} we obtain that
\begin{align}
\frac{d}{dr} \log (r^{-2\mu}h(r)) & \ge \frac{2}{r}\left(N_0(r) - \mu - Cr^{1 - b}(N_0(r) + 1)\right) \notag \\
& = \frac{2}{r}\left((1 - Cr^{1 - b})N_0(r) - (1 - Cr^{1 - b})\mu - Cr^{1 - b}(1 + \mu)\right) \notag \\
& \ge \frac{2}{r}\left(- \frac{(1 - Cr^{1 - b})Cr(1 + \mu)^2}{1 + Cr(1 + \mu)} - Cr^{1 - b}(1 + \mu)\right). \label{loghLB}
\end{align}
Hence, we conclude that 
\[
\frac{d}{dr} \log (r^{-2\mu}h(r)) \ge - C
\]
for $\mathcal{L}^1$-a.e.\@ $r < r_0$, where $C$ is a constant that depends on both the frequency $\mu$ and  on $\|u\|_{L^{\infty}(B_{r_0}')}^{p - 2}$ when $p > 2$. Integrating this inequality over $(r, R) \subset (0, r_0)$ yields 
\[
r^{-2\mu} h(r) \le R^{-2\mu} h(R) e^{C(R - r)}.
\]
To prove the second part of the statement, observe that for every $\delta > 0$ there exists $r_{\delta}$ such that $N_0(r) - \mu \le \delta/4$ for all $r<r_\delta$. Consequently, we obtain that
\[
\frac{d}{dr} \log (r^{-2\mu}h(r)) \le \frac{2}{r}\left(\frac{\delta}{4} + Cr^{1 - b}(\delta + \mu + 1)\right).
\]
Hence,  possibly replacing $r_{\delta}$ with a smaller number, we conclude that
\[
\frac{d}{dr} \log (r^{-2\mu}h(r)) \le \frac{\delta}{r}.
\]
As above, integrating this inequality over $(r, R) \subset (0, r_{\delta})$ yields
\[
R^{-2\mu} h(R) \left(\frac{r}{R}\right)^{\delta} \le r^{-2\mu} h(r).
\]
This concludes the proof.
\end{proof}

\begin{cor} 
Under the assumptions of \Cref{double-lem}, there exists a constant $C$ such that for every $r < R$ we have that 
\[
\max\left\{\inf_{S_r^+}|u|, \inf_{S_r^+}|v|\right\} \le Cr^{\mu}.
\]
\end{cor}
\begin{proof}
Arguing by contradiction, assume that there exist a sequence $r_j \to 0^+$ and constants $C_j \to \infty$ such that 
\[
\max \left\{\inf_{S_{r_j}^+} |u|, \inf_{S_{r_j}^+} |v|\right\} \ge C_j r_j^{\mu}.
\]
Then we have 
\[
r_j^{-2\mu} h(r_j) \ge \frac{1}{r_j^{n + b + 2\mu}} \int_{S_{r_j}^+} y^b \left(\max \left\{\inf_{S_{r_j}^+} |u|, \inf_{S_{r_j}^+} |v|\right\}\right)^2 \,d\Hh^n \ge C(n, b) C_j^2.
\]
Letting $j \to \infty$ gives a contradiction with \Cref{double-lem} $(i)$.
\end{proof}

We conclude the section by showing that a version of statement $(i)$ in \Cref{double-lem} continues to hold even when the frequency is infinite. This result plays an important role later in our study of the free boundary (see, in particular, the proof of \Cref{Fsig}).
\begin{lem}
\label{double-infty}
Under the assumptions of \Cref{AM-thm}, let $h$ be as in \eqref{h-def} and $\mu$ as in \eqref{frequency_def}. Assume that $\mu = \infty$. Then, for every $M > 0$ there exists $r_M > 0$ such that for every $0 < r < R < r_M$, we have
\[
r^{-2M}h(r) \le R^{-2M}h(R)e^{C(R - r)}
\]
Here $C$ is a constant that possibly depends on $\|u\|_{L^{\infty}(B_{r_0}')}^{p - 2}$ when $p > 2$.
\end{lem}
\begin{proof}
Let $C$ be as in \Cref{P-lem}. Then, for every $M > 0$ there exists $r_M > 0$ such that 
\begin{equation}
\label{mu>M}
(1 - Cr^{1 - b})N_0(r) \ge M
\end{equation}
for all $0 < r < r_M$. Then in view of \eqref{mu>M}, arguing as in the proof of \Cref{double-lem} and by eventually replacing $r_M$ with a smaller number (see in particular \eqref{logh} and \eqref{loghLB}), we obtain that 
\begin{align*}
\frac{d}{dr} \log (r^{-2M}h(r)) & = \frac{2}{r}\left(N_0(r) - M + \frac{rP(r)}{H(r)}\right) \\
& \ge \frac{2}{r}\left(N_0(r) - M - Cr^{1 - b}(N_0(r) + 1)\right) \\
& \ge \frac{2}{r}\left((1 - Cr^{1 - b})N_0(r) - M\right) - C \\
& \ge -C
\end{align*}
holds for $\mathcal{L}^1$-a.e.\@ $r \in (0, r_M)$.
The rest follows by integrating the resulting differential inequality over the interval $(r, R) \subset (0, r_M)$. 
\end{proof}

\section{Almgren's blow-up solutions}
\label{ABU-sec}
We recall that, in view of the local character of the analysis we employed, the conclusions of the previous section continue to hold for (sufficiently small) balls centered at $(x_0, 0)$, for any $x_0 \in B'_1$. Throughout this section, we continue to assume $x_0 = 0$ and $b\leq 0$.

Let $h$ be given as in \eqref{h-def} and, for $X \in B_1^+$, define the Almgren's rescalings
\begin{equation}
\label{A-rescaling}
u_r(X) \coloneqq \frac{u(r X)}{h(r)^{1/2}} \qquad \text{ and } \qquad v_r(X) \coloneqq \frac{v(r X)}{h(r)^{1/2}}.
\end{equation}
In the following we always assume that $r$ is sufficiently small, so that if $u$ and $v$ are not identically zero, then $h(r) > 0$ (see \Cref{H>0}). 

The main result of this section is the study of blow-up profiles for $\{u_r\}_r$ and $\{v_r\}_r$.

\begin{thm}
\label{BUL}
Let $u$ be a weak solution to \eqref{bil-noBC} in the sense of \Cref{ws-def}. Let $b \le 0$, and let $N_0$ be given as in \eqref{N0-def} and $\mu$ as in \eqref{frequency_def}. Suppose $f$ satisfies \eqref{H2'}, \eqref{H4'}, and \eqref{H5'}. Assume that $u$ is not identically zero, and let the rescaled families $\{u_r\}_r$ and $\{v_r\}_r \subset H^1(B_1^+; y^b)$ be defined as in \eqref{A-rescaling}. Finally, assume that $\mu < \infty$ and fix $R > 0$. Then for every sequence $r_j \to 0^+$ there exist a subsequence (which we do not relabel) and functions $\tilde{u}, \tilde{v} \in H^1(B_R^+; y^b)$ such that 
\begin{align} 
u_{r_j} \to \tilde{u}, \quad v_{r_j} & \to \tilde{v} \qquad \text{ in } H^1(B_R^+; y^b), \label{bulH} \\
u_{r_j} \to \tilde{u}, \quad v_{r_j} & \to \tilde{v} \qquad \text{ in } C^{1, \alpha}(B_R^+), \label{bulC}
\end{align}
as $j \to \infty$, for some $\alpha \in (0, 1)$. Moreover, $\mu \in \NN \cup \{0\}$, and $\tilde{u}$ and $\tilde{v}$ are homogeneous functions of degree $\mu$ that cannot simultaneously be identically equal to zero and satisfy 
\[
\left\{
\arraycolsep=1.4pt\def\arraystretch{1.6}
\begin{array}{rll}
\mathcal{L}_b(\tilde{u}) = & 0 & \text{ in } B_R^+, \\
\mathcal{L}_b(\tilde{v}) = & 0 & \text{ in } B_R^+, \\
\partial_y^b \tilde{u} = & 0 & \text{ on } B_R', \\
\partial_y^b \tilde{v} = & 0  & \text{ on } B_R'.
\end{array}
\right.
\]
\end{thm}

\begin{proof} We divide the proof into two steps.
\newline
\textbf{Step 1:} Fix $R_0 > \min\{1, R\}$ and assume that $r > 0$ is sufficiently small so that $rR_0 < \min\{r_0, r_1\}$, where $r_0$ and $r_1$ are given as in \Cref{H>0} and \Cref{double-lem} $(ii)$, respectively. In particular, we can assume that $H(rR_0) > 0$, that $N_0(rR_0, u, v) \le \mu + 1$, and that
\begin{equation}
\label{rR0/r}
\frac{h(rR_0)}{h(r)} \le R_0^{2\mu + 1}.
\end{equation}
Next, by a change of variables and \eqref{A-rescaling}, we obtain that
\begin{align*}
\int_{B_{rR_0}^+} y^b(|\nabla u|^2 + |\nabla v|^2)\,dX & = r^{n - 1 + b} h(r) \int_{B_{R_0}^+} y^b(|\nabla u_r|^2 + |\nabla v_r|^2)\,dX, \\
\int_{S_{rR_0}^+} y^b(u^2 + v^2)\,d\Hh^n & = r^{n + b} h(r) \int_{S_{R_0}^+} y^b(u_r^2 + v_r^2)\,d\Hh^n.
\end{align*}
Consequently, we have that
\begin{equation}
\label{NrR0}
N_0(rR_0, u, v) = \frac{rR_0 D_0(rR_0, u, v)}{H(rR_0, u, v)} = \frac{R_0 D_0(R_0, u_r, v_r)}{H(R_0, u_r, v_r)} = N_0(R_0, u_r, v_r),
\end{equation}
and moreover, by \eqref{rR0/r}, it follows that
\begin{equation}
\label{HR0-bound}
H(R_0, u_r, v_r) = R_0^{n + b}\frac{h(rR_0)}{h(r)} \le R_0^{n + b + 2\mu + 1}.
\end{equation}
Combining \eqref{NrR0} and \eqref{HR0-bound}, we conclude that
\begin{equation}
\label{rR0-grad}
\int_{B_{R_0}^+} y^b(|\nabla u_r|^2 + |\nabla v_r|^2)\,dX = \frac{1}{R_0}N_0(rR_0, u, v)H(R_0, u_r, v_r) \le (\mu + 1)R_0^{n + b + 2\mu}.
\end{equation}
Additionally, by \Cref{PI}, together with \eqref{HR0-bound} and \eqref{rR0-grad}, we obtain
\[
\int_{B_{R_0}^+} y^b(u_r^2 + v_r^2)\,dX \le C(n, b, \mu)R_0^{n + b + 2\mu + 2}.
\]
This shows that the families $\{u_r\}_r$ and $\{v_r\}_r$ are bounded in $H^1(B_{R_0}^+; y^b)$, provided that $r$ is sufficiently small. In particular, for every sequence $\{r_j\}_j$, $j \ge 2$, such that $r_j \to 0^+$, there exist a subsequence (which we do not relabel) and functions $\tilde{u}, \tilde{v} \in H^1(B_{R_0}^+; y^b)$ such that $u_{r_j} \to \tilde{u}$ and $v_{r_j} \to \tilde{v}$ strongly in $L^2(B_{R_0}^+; y^b)$ (see, for example, \cite{MR1455468}), strongly in $L^2(S_1^+; y^b)$ (see \Cref{TO}), and weakly in $H^1(B_{R_0}^+; y^b)$. Throughout the remainder of the proof, we continue to work with this subsequence; however, we choose to omit the subscript $j$ to keep the notation as simple as possible. In particular, keeping in mind the definition of $h(r)$ \eqref{h-def}, the convergence of the traces obtained above implies that
\begin{equation}
\label{bul=1}
\int_{S_1^+} y^b(\tilde{u}^2 + \tilde{v}^2)\,d\Hh^n = \lim_{r \to 0^+} \int_{S_1^+} y^b(u_r^2 + v_r^2)\,d\Hh^n = 1,
\end{equation}
thus proving that $\tilde{u}, \tilde{v}$ are not both identically zero.
\newline
\textbf{Step 2:} We now improve the convergence obtained in the previous step. To this end, notice that $u_r$ and $v_r$ are weak solutions to 
\[
\left\{
\arraycolsep=1.4pt\def\arraystretch{1.6}
\begin{array}{rll}
\Delta_b u_r = & r^2v_r & \text{ in } B_{R_0}^+, \\
\Delta_b v_r = & 0 & \text{ in } B_{R_0}^+, \\
\partial_y^b u_r = & 0 & \text{ on } B_{R_0}', \\
\partial_y^b v_r = &  \tilde{f}(\cdot, u_r) & \text{ on } B_{R_0}',
\end{array}
\right.
\]
where 
\[
\tilde{f}(x, u_r(x, 0)) \coloneqq \frac{r^{1 - b}}{h(r)^{1/2}}f(rx, h(r)^{1/2}u_r(x, 0)).
\]
Observe that by \eqref{H2'}, \Cref{double-lem} $(i)$, and using that by assumption $p \ge 2$, there exists a constant $C$ (independent of $r$) such that
\begin{equation}
\label{tildeH2}
|\tilde{f}(x, u_r(x, 0))| \le C r^{1 - b} |u_r(x, 0)|^{p - 1}.
\end{equation}
Furthermore, by following the circle of ideas in  the proof of \Cref{Lq-reg}, for every $q \in [1, \infty)$ and every $R_1 \in (R, R_0)$, there exists a constant $C = C(n, b, q, R_0, R_1)$ such that 
\begin{align*}
\|u_r\|_{L^q(B_{R_1}^+; y^b)} + \|u_r\|_{L^q(B_{R_1}')} & \le C \left(\|u_r\|_{H^1(B_{R_0}^+; y^b)} + \|v_r\|_{H^1(B_{R_0}^+; y^b)}\right), \\
\|v_r\|_{L^q(B_{R_1}^+; y^b)} + \|v_r\|_{L^q(B_{R_1}')} & \le C \left(\|u_r\|_{H^1(B_{R_0}^+; y^b)} + \|v_r \|_{H^1(B_{R_0}^+; y^b)}\right). 
\end{align*}
In view of \eqref{tildeH2}, combining these estimates with the results of the previous step we find that for every $q \in [1, \infty)$ there exists a constant $C = C(n, b, q, R_0, R_1, \mu)$ such that 
\begin{equation}
\label{stv-ur}
\|v_r\|_{L^q(B_{R_1}^+)} + \| \tilde{f}(\cdot, u_r(\cdot, 0)) \|_{L^q(B_{R_1}')} \le C.
\end{equation}
Now, if $b < 0$ and $q$ is chosen to be sufficiently large, \eqref{stv-ur} allows us to apply \Cref{STV1.5} and \Cref{STV} $(iii)$ to conclude that $\{u_r\}_r$ and $\{v_r\}_{r}$ are bounded in $C^{1, \alpha}(B_R)$ for some $\alpha \in (0, 1)$. Similarly, if $b = 0$, fix $R_2 \in (R, R_1)$. Then \Cref{STV1.5} and \Cref{STV} $(iii)$ can be applied to show that $\{v_r\}_r$ is bounded in $C^{0, \alpha}(B_{R_2}^+)$ and that $\{u_r\}_r$ is bounded in $C^{1, \alpha}(B_{R_2}^+)$, respectively. To proceed, arguing as in the proof of \Cref{improved-reg}, we further deduce that $\{v_r\}_r$ is also bounded in $C^{1, \alpha}(B_R^+)$. Thus, \eqref{bulH} and \eqref{bulC} follow immediately in this case as well.
\newline
\textbf{Step 3:} Finally, observe that by Corollary \ref{v-solves} and \eqref{tildeH2}, for every $\phi \in H^1_{0, S_R^+}$ we have 
\begin{equation}
\label{weakvtilde}
\left| \int_{B_R^+} y^b \nabla v_r \cdot \nabla \phi\,dX \right| \le \int_{B_R'} |\tilde{f}(x, u_r(x, 0))\phi(x, 0)|\,dx \le C r^{1 - b} \int_{B_R'}|u_r(x, 0)|^{p - 1}|\phi(x, 0)|\,dx.
\end{equation}
Therefore, letting $r \to 0^+$ in \eqref{weakvtilde}, we obtain that 
\[
\int_{B_R^+} y^b \nabla \tilde{v} \cdot \nabla \phi \,dX = 0.
\]
Similarly, passing to the limit in the weak formulation for equation satisfied by $u_r$, shows that
\[
\int_{B_R^+} y^b \nabla \tilde{u} \cdot \nabla \phi \,dX = 0
\]
must hold for every $\phi \in H^1_{0, S_R^+}$. In particular, $\tilde{u}$ and $\tilde{v}$ are smooth by \Cref{STV} $(i)$ and computations analogous to those in the proof of \Cref{D_0H-der} yield
\begin{equation}
\label{D_0'tilde}
\frac{d}{d\rho}D_0(\rho, \tilde{u}, \tilde{v}) = \frac{n + b - 1}{\rho}D_0(\rho, \tilde{u}, \tilde{v}) + 2 \int_{S_{\rho}^+} y^b(\tilde{u}_{\nu}^2 + \tilde{v}_{\nu}^2)\,d\mathcal{H}^n.
\end{equation} 
Next, observe that arguing as in \eqref{NrR0}, if $\rho \le R$ we have that
\begin{equation}
\label{N0constant}
\mu = \lim_{r \to 0^+} N_0(r\rho, u, v) = \lim_{r \to 0^+} N_0(\rho, u_r, v_r) = N_0(\rho, \tilde{u}, \tilde{v}),
\end{equation}
where in the last equality we have used \eqref{bulH}. Following the computations in Step 1 in the proof of \Cref{AM-thm}), by \eqref{D_0'tilde} and \eqref{N0constant} we obtain that 
\begin{align*}
0 = \frac{d}{d\rho}N_0(\rho, \tilde{u}, \tilde{v}) & = -(n + b - 1)\frac{D_0(\rho, \tilde{u}, \tilde{v})}{H(\rho, \tilde{u}, \tilde{v})} + \frac{\rho D_0'(\rho, \tilde{u}, \tilde{v})}{H(\rho, \tilde{u}, \tilde{v})} - \frac{2\rho D_0(\rho, \tilde{u}, \tilde{v})^2}{H(\rho, \tilde{u}, \tilde{v})^2} \\
& = 2\rho\frac{ \int_{S_{\rho}^+}y^b(\tilde{u}_{\nu}^2 + \tilde{v}_{\nu}^2)\,d\mathcal{H}^n}{\int_{S_{\rho}^+}y^b(\tilde{u}^2 + \tilde{v}^2)\,d\mathcal{H}^n} - 2\rho \left(\frac{\int_{S_{\rho}^+} y^b(\tilde{u}\tilde{u}_{\nu} + \tilde{v}\tilde{v}_{\nu})\,d\mathcal{H}^n}{\int_{S_{\rho}^+}y^b(\tilde{u}^2 + \tilde{v}^2)\,d\mathcal{H}^n}\right)^2.
\end{align*}
In particular, we have that
\begin{equation}
\label{CS=}
\left( \int_{S_{\rho}^+} y^b (\tilde{u}^2 + \tilde{v}^2)\,d\Hh^n\right) \left(\int_{S_{\rho}^+} y^b(\tilde{u}_{\nu}^2 + \tilde{v}_{\nu}^2)\,d\Hh^n \right) - \left(\int_{S_{\rho}^+} y^b(\tilde{u} \tilde{u}_{\nu} + \tilde{v} \tilde{v}_{\nu})\ d\Hh^n\right)^2\, = 0.
\end{equation}
Since \eqref{CS=} corresponds to the equality case in the Cauchy-Schwarz inequality, the desired result follows as in the proof of Lemma 5.1 in \cite{MR4169657}. This concludes the proof.
\end{proof}

\begin{rmk} 
\label{Pmu-def}
It is worth noting that arguing exactly as in the proof of Lemma 5.1 in \cite{MR4169657} yields a finer description of the blow-up profiles found in \Cref{BUL}. In particular, we have that 
\begin{equation}
\label{lincomb}
\tilde{u}(X) = |X|^{\mu} \Psi_1(X/|X|) \qquad  \text{ and } \qquad \tilde{v}(X) = |X|^{\mu}\Psi_2(X/|X|),
\end{equation}
where $\Psi_1$ and $\Psi_2$ are linear combinations of eigenfunctions of 
\[
\left\{
\arraycolsep=1.4pt\def\arraystretch{1.6}
\begin{array}{rll}
-\di(y^b \nabla_{S_1^+}\Psi) = & \lambda y^b \Psi & \text{ in } S_1^+, \\
\partial_y^b \Psi = & 0 & \text{ on } \partial S_1^+
\end{array}
\right.
\]
associated with the same eigenvalue $\lambda = \lambda(n, b, \mu)$. For this reason, in the sequel it will be convenient to denote by $\mathcal{P}_{\mu}$ the space consisting of all $\mathcal{L}_b$-harmonic functions that are symmetric with respect to the hyperplane $\{y = 0\}$ and admit a representation of the form \eqref{lincomb}. As one can readily check, every $w \in \mathcal{P}_{\mu}$ satisfies 
\[
\left\{
\arraycolsep=1.4pt\def\arraystretch{1.6}
\begin{array}{rll}
\mathcal{L}_b(w) = & 0 & \text{ in } B_R^+, \\
\partial_y^b w = & 0 & \text{ on } B_R',
\end{array}
\right.
\]
for some $R > 0$, and hence for all $R > 0$.
\end{rmk}

\section{Monotonicity formulas and homogeneous blow-up solutions}
\label{MW-sec}
In this section we prove a quasi-monotonicity formula (see \Cref{M-mon} and \Cref{AlmostMonneautonicity}) for the Monneau-type energy 
\begin{equation}
\label{Monneau-def}
M_{\mu} (r, u, v, p, q) \coloneqq \frac{1}{r^{n + b + 2\mu}}\int_{S_r^+}y^b\left((u - p)^2 + (v - q)^2\right)\,d\Hh^n,
\end{equation}
which we later use in \Cref{non-deg} to show the non-degeneracy of solutions. In \eqref{Monneau-def}, $p, q$ denote elements of $\mathcal{P}_{\mu}$ (see \Cref{Pmu-def}). Finally, we conclude the section with a proof of the existence and uniqueness of homogeneous blow-up solutions (see \Cref{HBUL} below).

Consider the Weiss-type functional 
\begin{equation}
\label{Weiss-def}
W_{\mu}(r, u, v) \coloneqq \frac{H(r, u, v)}{r^{n + b + 2\mu}}(N_0(r, u, v) - \mu).
\end{equation}
The following lemma shows that $W_{\mu}$ is invariant under translations by elements in $\mathcal{P}_{\mu}$.

\begin{lem}
\label{W-identity}
Let $u$ be a weak solution to \eqref{bil-noBC} in the sense of \Cref{ws-def}. Then, for every $p, q \in \mathcal{P}_{\mu}$, we have that
\[
W_{\mu}(r, u, v) = W_{\mu}(r, u - p, v - q).
\]
\end{lem}
\begin{proof}
Let $w_1 \coloneqq u - p$, $w_2 \coloneqq v - q$ and observe that 
\begin{align}
W_{\mu}(r, u, v) & = W_{\mu}(r, w_1 + p, w_2 + q) \notag \\
& = \frac{1}{r^{n + b + 2\mu}}(r D_0(r, w_1 + p, w_2 + q) - \mu H(r, w_1 + p, w_2 + q)). \label{W=}
\end{align}
By the divergence theorem and the fact that by homogeneity $p$ and $q$ satisfy $\nabla p \cdot X = \mu p$ and $\nabla q \cdot X = \mu q$, we obtain
\begin{align}
D_0(r, w_1 + p, w_2 + q) & = D_0(r, w_1, w_2) + 2 \int_{B_r^+} y^b(\nabla w_1 \cdot \nabla p + \nabla w_2 \cdot \nabla q)\,dX + D_0(r, p, q) \notag \\
& = D_0(r, w_1, w_2) + \frac{2 \mu}{r} \int_{S_r^+} y^b(w_1p + w_2 q)\,d\Hh^n + \frac{\mu}{r}H(r, p, q) \notag \\
& = D_0(r, w_1, w_2) + \frac{\mu}{r}(H(r, w_1 + p, w_2 + q) - H(r, w_1, w_2)). \label{D_0expand}
\end{align}
Finally, substituting \eqref{D_0expand} into \eqref{W=}, we find that 
\[
W_{\mu}(r, u, v) = \frac{1}{r^{n + b + 2\mu}}(r D_0(r, w_1, w_2) - \mu H(r, w_1, w_2)) = W_{\mu}(r, w_1, w_2).
\]
This concludes the proof.
\end{proof}
In the remainder of the section, we will assume $b\leq 0$.
\begin{prop}
\label{M-mon}
Let $u$ be a weak solution to \eqref{bil-noBC} in the sense of \Cref{ws-def}. Let $b \le 0$, $N_0$ be given as in \eqref{N0-def}, and  $\mu$ as in \eqref{frequency_def}. Suppose $f$ satisfies \eqref{H2'}, \eqref{H4'}, and \eqref{H5'}. Assume that $u$ is not identically zero, and that $\mu < \infty$. Finally, let $M_{\mu}$ and $W_{\mu}$ be given as in \eqref{Monneau-def} and \eqref{Weiss-def}, respectively, let $r_0$ be as in \Cref{H>0} and let $r_1$ be such that $N_0(r, u, v) \le \mu + 1$ for all $r < r_1$. Then
\[
\frac{d}{dr} M_{\mu}(r, u, v, p, q) \ge \frac{2}{r} W_{\mu}(r, u, v) - Cr^{-b}
\]
for $\mathcal{L}^1$-a.e.\@ $r \in (0, \min\{r_0, r_1\}/2)$, where $C$ is a positive constant that depends only on $n$, $b$, $\mu$, $r_0$, $r_1$, $\|u\|_{L^{\infty}(B_{r_0}^+)}$, $\|p\|_{L^2(S_1^+; y^b)}$, and $\|q\|_{L^2(\partial S_1^+)}$.
\end{prop}
\begin{proof}
As in the proof of \Cref{W-identity}, we let $w_1 \coloneqq u - p$ and $w_2 \coloneqq v - q$. Observe that
\[
M_{\mu}(r, u, v, p, q) = \frac{H(r, w_1, w_2)}{r^{n + b + 2\mu}}.
\]
Then, by \eqref{H'formula}, for $\mathcal{L}^1$-a.e.\@ $r > 0$ we have that
\begin{equation}
\label{M-derivative}
\frac{r}{2}\frac{d}{dr} M_{\mu}(r, u, v, p, q) = \frac{1}{r^{n + b + 2\mu}}(rD(r, w_1, w_2) - \mu H(r, w_1, w_2)).
\end{equation}
Moreover, using the identity $D = D_0 + P$, where $P$ is the perturbation functional defined in \eqref{P-def}, \Cref{W-identity} and \eqref{M-derivative} imply that 
\begin{equation}
\label{M-der-2}
\frac{r}{2}\frac{d}{dr} M_{\mu}(r, u, v, p, q) = W_{\mu}(r, u, v) + \frac{rP(r, w_1, w_2)}{r^{n + b + 2 \mu}}.
\end{equation}
To conclude, it remains to estimate the second term on the right-hand side of \eqref{M-der-2}. To this end, observe that 
\[
P(r, w_1, w_2) = \int_{B_r^+} y^b w_1 v \,dX + \int_{B_r'} f(x, u) w_2\,dx.
\]
Then, by following the argument in the proof of \Cref{P-lem} and recalling that $p, q \in \mathcal{P}_{\mu}$, we obtain
\begin{align}
\left|\int_{B_r^+}y^b w_1 v\,dX \right| & \le C \int_{B_r^+} y^b(u^2 + p^2 + v^2)\,dX \notag \\
& = Cr^{n + 1 + b + 2\mu} \int_{S_1^+}y^bp^2\,d\Hh^n + Cr(rD_0(r, u, v) + H(r, u, v)), \label{Pw1}
\end{align}
and
\begin{align}
\left|\int_{B_r'} f(x, u)w_2\,dx\right| & \le C \int_{B_r'} (v^2 + q^2 + u^2) \notag \\
& \le Cr^{n + 2\mu}\int_{\partial S_1^+} q^2\,d\Hh^{n - 1} + C r^{-b}(rD_0(r, u, v) + H(r, u, v)). \label{Pw2}
\end{align}
Combining \eqref{Pw1} and \eqref{Pw2} we get
\begin{align*}
\left| \frac{rP(r, w_1, w_2)}{r^{n + b + 2\mu}} \right| & \le \frac{Cr^{n + 2 \mu + 1} + Cr^{1-b}(rD_0(r, u, v) + H(r, u, v))}{r^{n + b + 2\mu}} \\
& = Cr^{1 - b} + Cr^{1-b}\frac{H(r, u, v)}{r^{n + b + 2\mu}}(N_0(r, u, v) + 1).
\end{align*}
In turn, by \Cref{double-lem} $(i)$, it follows that
\[
\left| \frac{rP(r, w_1, w_2)}{r^{n + b + 2\mu}} \right| \le Cr^{1 - b} +  Cr^{1 - b}R^{-2\mu}h(R)e^{C(R - r)} \le Cr^{1 - b}
\]
holds for $\mathcal{L}^1$-a.e.\@ $r < \min\{r_0, r_1\}/2$ and $R \in (\min\{r_0, r_1\}/2, r_0)$. This concludes the proof.
\end{proof}

\begin{cor}
\label{AlmostMonneautonicity} 
Under the assumptions of \Cref{M-mon}, we have that
\[
\frac{d}{dr}M_{\mu}(r, u, v, p, q) \ge -C
\]
holds for $\mathcal{L}^1$-a.e.\@ $r \in (0, \min\{r_0, r_1\}/2)$. Moreover, 
\[
M_{\mu}(0^+, u, v, p, q) \coloneqq \lim_{r \to 0^+} M_{\mu}(r, u, v, p, q)
\] 
exists and is finite.
\end{cor}
\begin{proof}
Arguing as in \eqref{N0-mu0}, we see that
\[
N_0(r) - \mu \ge - \frac{Cr(1 + \mu)^2}{1 + Cr(1 + \mu)} \ge - Cr,
\]
possibly increasing the value of $C$ as needed.
Consequently, applying \Cref{double-lem} $(i)$, we obtain
\[
W_{\mu}(r, u, v) = \frac{H(r)}{r^{n + b + 2\mu}}(N_0(r) - \mu) \ge -Cr \frac{h(r)}{r^{2\mu}} \ge -Cr \frac{h(R)}{R^{2\mu}}e^{C(R - r)} = -Cr.
\]
Since $b \le 0$, the desired result readily follows from \Cref{M-mon}.
\end{proof}

\begin{prop}
\label{non-deg}
Under the assumptions of \Cref{M-mon}, let $r_0$ be given as in \Cref{H>0}. Then there exists a positive constant $c$ such that for every $r < r_0$ we have that 
\[
h(r) \ge cr^{2 \mu}.
\]
\end{prop}
\begin{proof}
Arguing by contradiction, assume that there exist a sequence $\{r_j\}_j \subset (0, r_0)$ and constants $C_j \to 0$ such that 
\begin{equation}
\label{h/0}
\frac{h(r_j)}{r_j^{2\mu}} \le C_j \to 0.
\end{equation}
Observe that, by \Cref{H>0}, we can assume without loss of generality that $r_j \to 0^+$, since otherwise \eqref{h/0} would immediately yield a contradiction. Let $\{u_{r_j}\}_j$, $\{v_{r_j}\}_j$ be defined as in \eqref{A-rescaling}. Eventually extracting a subsequence (which we do not relabel), we can find $\tilde{u}, \tilde{v} \in \mathcal{P}_{\mu}$ as in \Cref{BUL}. Observe that by H\"older's inequality and \eqref{bul=1} we have that 
\begin{equation}
\label{holderM} 
\frac{h(r_j)}{r_j^{2\mu}} - \frac{2h(r_j)^{1/2}}{r_j^{\mu}} + 1 \le M_{\mu}(r_j, u, v, \tilde{u}, \tilde{v}) \le \frac{h(r_j)}{r_j^{2\mu}} + \frac{2h(r_j)^{1/2}}{r_j^{\mu}} + 1.
\end{equation}
Letting $j \to \infty$ in \eqref{holderM}, it follows from \eqref{h/0} that
\begin{equation}
\label{hom-lim}
M_{\mu}(0^+, u, v, \tilde{u}, \tilde{v}) = \lim_{j \to \infty} M_{\mu}(r_j, u, v, \tilde{u}, \tilde{v}) = 1.
\end{equation}
Next, observe that by \Cref{AlmostMonneautonicity} we must have that
\begin{equation}
\label{ftcM}
- Cr \le M_{\mu}(r, u, v, \tilde{u}, \tilde{v}) - M_{\mu}(0^+, u, v, \tilde{u}, \tilde{v}) = \frac{h(r)}{r^{2\mu}} - \frac{2 h(r)^{1/2}}{r^{\mu}}\int_{S_1^+} y^b(u_r \tilde{u} + v_r \tilde{v})\,d\Hh^n,
\end{equation}
where the last equality is a consequence of \eqref{A-rescaling}, \eqref{hom-lim}, and the homogeneity of $\tilde{u}$ and $\tilde{v}$. In particular, setting $r = r_j$ in \eqref{ftcM} and applying \Cref{double-lem} $(ii)$ for some $\delta \in (0, 1)$ yields
\[
-C r_j^{1 - \delta/2} \le \frac{h(r_j)^{1/2}}{r_j^{\mu}} - 2 \int_{S_1^+} y^b(u_{r_j} \tilde{u} + v_{r_j} \tilde{v})\,d\Hh^n.
\]
Once again letting $j \to \infty$ in the previous inequality yields
\[
0 \le - 2 \int_{S_1^+} y^b(\tilde{u}^2 + \tilde{v}^2)\,d\Hh^n = -2.
\]
Since we have reached a contradiction, the proof is complete. 
\end{proof}

\begin{cor}
Under the assumptions of \Cref{M-mon}, there exist a positive constant $c$ and $R > 0$ such that for every $r < R$ we have that 
\[
\max\left\{ \sup_{S_r^+} |u|, \sup_{S_r^+} |v| \right\} \ge cr^{\mu}. 
\]
\end{cor}
\begin{proof}
Arguing by contradiction, assume that there exist a sequence $r_j \to 0^+$ and constants $C_j \to 0$ such that 
\[
\frac{1}{r_j^{\mu}} \sup_{S_r^+} |u| \le C_j \qquad \text{ and } \qquad \frac{1}{r_j^{\mu}} \sup_{S_r^+} |v| \le C_j.
\]
In turn, we must have that
\[
\frac{h(r_j)}{r_j^{2\mu}} = \frac{1}{r_j^{n + b + 2\mu}} \int_{S_{r_j}^+} y^b(u^2 + v^2)\,d\Hh^n \le C(n, b)C_j.
\]
Letting $j \to \infty$ gives a contradiction with \Cref{non-deg}.
\end{proof}

Next, consider the \emph{homogeneous rescaling} of $u$ and $v$ (at the origin), which are defined by 
\begin{equation}
\label{hbu-def}
u_{\mu, r}(X) \coloneqq \frac{u(rX)}{r^{\mu}}, \qquad v_{\mu, r}(X) \coloneqq \frac{v(rX)}{r^{\mu}}.
\end{equation}

\begin{thm}
\label{HBUL}
Let $u$ be a weak solution to \eqref{bil-noBC} in the sense of \Cref{ws-def}. Let $b \le 0$,  $N_0$ be given as in \eqref{N0-def}, and  $\mu$ as in \eqref{frequency_def}. Suppose $f$ satisfies \eqref{H2'}, \eqref{H4'}, and \eqref{H5'}. Assume that $u$ is not identically zero, and let the rescaled families $\{u_{\mu, r}\}_r$ and $\{v_{\mu, r}\}_r \subset H^1(B_1^+; y^b)$ be defined as in \eqref{hbu-def}. Finally, assume that $\mu < \infty$ and fix $R > 0$. Then there exist $\tilde{u}_{\mu}, \tilde{v}_{\mu} \in \mathcal{P}_{\mu}$ such that 
\begin{align} 
u_{\mu, r} \to \tilde{u}_{\mu}, \quad v_{\mu, r} & \to \tilde{v}_{\mu} \qquad \text{ in } H^1(B_R^+; y^b), \label{hbulH} \\
u_{\mu, r} \to \tilde{u}_{\mu}, \quad v_{\mu, r} & \to \tilde{v}_{\mu} \qquad \text{ in } C^{1, \alpha}(B_R^+), \label{hbulC}
\end{align}
for some $\alpha \in (0, 1)$ as $r \to 0^+$. Additionally, $\tilde{u}_{\mu}$ and $\tilde{v}_{\mu}$ cannot simultaneously be identically equal to zero.
\end{thm}

\begin{proof}
Observe that by \Cref{double-lem} $(i)$ and \Cref{non-deg}, for all $r < r_0$, we have that 
\begin{equation}
\label{h-scale}
cr^{2 \mu} \le h(r) \le Cr^{2\mu},
\end{equation}
for some constants $c, C > 0$. Let $\{r_j\}_j$ be a vanishing sequence. Then, eventually extracting a subsequence (which we do not relabel), we can find $\tilde{u}, \tilde{v} \in \mathcal{P}_{\mu}$ as in \Cref{BUL} and a positive constant $\gamma$ such that $r_j^{-2\mu}h(r_j) \to \gamma$. In turn, if we set 
\[
\tilde{u}_{\mu} \coloneqq \sqrt{\gamma} \tilde{u} \qquad \text{ and } \qquad \tilde{v}_{\mu} \coloneqq \sqrt{\gamma} \tilde{v},
\]
we readily obtain that \eqref{hbulH} and \eqref{hbulC} hold for $r = r_j$. Observe that $\tilde{u}_{\mu}$ and $\tilde{v}_{\mu}$ cannot both be identically zero, since they are multiples of the homogeneous functions $\tilde{u}$ and $\tilde{v}$ obtained in \Cref{BUL}, which themselves are nontrivial.

To conclude the proof, it remains to show that the homogenous blow-up solutions $\tilde{u}_{\mu}$ and $\tilde{v}_{\mu}$ do not depend on our choice of the sequence $\{r_j\}_j$. To this end, let $r_j, R_j \to 0^+$ be given and assume that $u_{\mu, r_j} \to \tilde{u}_{\mu}$, $u_{\mu, R_j} \to \tilde{U}_{\mu}$, and similarly $v_{\mu, r_j} \to \tilde{v}_{\mu}$ and $v_{\mu, R_j} \to \tilde{V}_{\mu}$. Then by \Cref{AlmostMonneautonicity} we have that  
\begin{align}
M_{\mu}(0^+, u, v, \tilde{u}_{\mu}, \tilde{v}_{\mu}) & = \lim_{j \to \infty} M_{\mu}(r_j, u, v, \tilde{u}_{\mu}, \tilde{v}_{\mu}) \notag \\
& = \lim_{j \to \infty} M_{\mu}(1, u_{\mu, r_j}, v_{\mu, r_j}, \tilde{u}_{\mu}, \tilde{v}_{\mu}) = 0, \label{M1=0}
\end{align}
and we similarly obtain 
\begin{equation}
\label{M2=0}
M_{\mu}(0^+, u, v, \tilde{U}_{\mu}, \tilde{V}_{\mu}) = \lim_{j \to \infty} M_{\mu}(1, u_{\mu, R_j}, v_{\mu, R_j}, \tilde{U}_{\mu}, \tilde{V}_{\mu}) = 0.
\end{equation}
Finally, combining \eqref{M1=0} and \eqref{M2=0} shows that 
\begin{align*}
\int_{S_1^+} y^b((\tilde{u}_{\mu} - \tilde{U}_{\mu})^2 + (\tilde{v}_{\mu} - \tilde{V}_{\mu})^2)\,d\Hh^n & = M_{\mu}(r,  \tilde{u}_{\mu}, \tilde{v}_{\mu}, \tilde{U}_{\mu}, \tilde{V}_{\mu}) \\
& \le 2M_{\mu}(r, u, v, \tilde{u}_{\mu}, \tilde{v}_{\mu}) + 2 M_{\mu}(r, u, v, \tilde{U}_{\mu}, \tilde{V}_{\mu}) \to 0
\end{align*}
as $r \to 0^+$. From this, it readily follows that $\tilde{u}_{\mu} = \tilde{U}_{\mu}$ and $\tilde{v}_{\mu} = \tilde{V}_{\mu}$, and therefore the proof is complete.
\end{proof}

As one can readily check, the proof of \Cref{BU} follows by combining the results of \Cref{AM-thm} and \Cref{HBUL}.

\section{The free boundary: regularity and structural properties}
\label{FB-sec}
In this section we undertake the study of the free boundary for non-trivial solutions to \eqref{bil-noBC}. We begin by recalling that in the context of this paper, the free boundary is the interface that separates the two phases $\{u > 0\}$ and $\{u < 0\}$ in the thin space. To be precise, the free boundary is defined as
\[
F(u) \coloneqq \left(\partial\{ u > 0 \} \cup \partial\{ u < 0\}\right) \cap B_1';
\]
its properties are classically analyzed by splitting $F(u)$ into two disjoint subsets, namely
\[
R(u) \coloneqq \{ X \in F(u) : \nabla_x'u(X) \neq 0 \}
\]
and
\[
S(u) \coloneqq F(u) \setminus R(u).
\]
We will refer to these subsets as the \emph{regular} and the \emph{singular} part of $F(u)$, respectively. 

\begin{thm}
Let $u$ be a weak solution to \eqref{bil-noBC} in the sense of \Cref{ws-def}. Then, in a neighborhood of any point in $R(u)$, the free boundary $F(u)$ coincides with the graph of a $C^{3, \alpha}$ function if $b \leq 0$ and a $C^{2, \alpha}$ function if $b > 0$.
\end{thm}

\begin{proof}
Let $(x, 0) \in R(u)$ be given. Then, by definition of $R(u)$ there exists $i \in \{1, \dots, n\}$ such that $\partial_{x_i} u(x, 0) \neq 0$. Without loss of generality, assume that $i = n$. But then, by the Implicit Function Theorem, there exist open sets $U \subset \RR^{n - 1}$ and $V \subset \RR$ with $x \in U \times V$ and a function $\psi \colon U \to V$ such that 
\[
(x_1, \dots, x_{n - 1}, \psi(x_1, \dots, x_{n - 1})) = x,
\]
and with the property that
\[
F(u) \cap (U \times V \times \{0\}) = \{(\hat{x}, \psi(\hat{x}), 0): \hat{x} \in U \}.
\] 
In particular, $F(u)$ is locally described by the graph of $\psi$. Since $\psi$ inherits the regularity of $u$, the conclusion follows from \Cref{Reg-cor} and \Cref{Reg-cor-0}. 
\end{proof}

The remainder of the section focuses on the study of the singular part when $b\leq 0$. Given $X_0 \in S(u)$, we use $N_0^{X_0}(r)$ to denote the Almgren frequency functional centered at $X_0$ (see \eqref{N0-def}); to be precise, 
\begin{equation}
\label{N_0shift-def}
N_0^{X_0}(r) \coloneqq N_0(r, u(X_0 + \cdot), v(X_0 + \cdot)).
\end{equation}
We recall that $N_0^{X_0}(0^+)$ exists by \Cref{AM-thm}; moreover, if the limit is finite, then it must be a non-negative integer by \Cref{BUL}. We then define
\[
S_{\mu}(u) = \{X_0 \in S(u) : N_0^{X_0}(0^+) = \mu\}, \qquad \mu \in \NN \cup \{0, \infty\}.
\]

\begin{lem}
\label{Fsig}
Let $u$ be a weak solution to \eqref{bil-noBC} in the sense of \Cref{ws-def}. Let $b \le 0$, $\mu$ as in \eqref{frequency_def}, and suppose $f$ satisfies \eqref{H2'}, \eqref{H4'}, and \eqref{H5'}. Assume that $u$ is not identically zero and that $\mu < \infty$. Then $S_{\mu}(u)$ is of type $F_{\sigma}$, that is, $S_{\mu}(u)$ is the union of countably many closed sets.
\end{lem}
\begin{proof}
For $j \ge 2$ and $r_0$ as in \Cref{double-lem} $(i)$, let 
\[
E_j \coloneqq \left\{X_0 \in S_{\mu}(u) \cap \overline{B_{1 - \frac{1}{j}}} : \frac{1}{j}r^{2\mu} \le h^{X_0}(r) \le j r^{2\mu}, 0 < r < \min\{r_0, 1 - |X_0|\} \right\},
\]
where, similarly to \eqref{N_0shift-def}, we define 
\[
h^{X_0}(r) \coloneqq h(r, u(X_0 + \cdot), v(X_0 + \cdot)) = \frac{1}{r^{n + b}} \int_{S_r^+} y^b(u(X_0 + X)^2 + v(X_0 + X)^2)\,d\Hh^n.
\]
It is straightforward to verify that (see in particular \eqref{h-scale})
\[
S_{\mu}(u) = \bigcup_{j = 2}^{\infty} E_j;
\]
thus it remains to show that $E_j$ is closed for every $j$. To this end, let $\overline{X} \in \overline{E_j}$ and let $\{X_m\}_m \subset E_j$ be such that $X_m \to \overline{X}$ as $m \to \infty$. As one can readily check, for every $\e \in (0, 1 - |\overline{X}|)$ and every 
\begin{equation}
\label{rminbound}
r < \min\{r_0, 1 - |\overline{X}| - \e\},
\end{equation} 
we have that $h^{X_m}(r) \to h^{\overline{X}}(r)$. Therefore, letting $m \to \infty$, we obtain that 
\begin{equation}
\label{non-degx0}
\frac{1}{j}r^{2\mu} \le h^{\overline{X}}(r) \le j r^{2\mu} 
\end{equation}
holds for all $r$ as in \eqref{rminbound}. Since $\e$ is arbitrary, the inequality continues to hold for $0 < r < \min\{r_0, 1 - |\overline{X}|\}$. Consequently, to prove that $\overline{X} \in E_j$, it remains to show that $\overline{X} \in S_{\mu}(u)$ (or equivalently, that $N_0^{\overline{X}}(0^+) = \mu$). Arguing by contradiction, assume that $N_0^{\overline{X}}(0^+) = \overline{\mu} < \mu$. Then, by \Cref{non-deg}, we get that for all $r$ sufficiently small
\[
cr^{2\overline{\mu}} \le h^{\overline{X}}(r).
\]
However, this is in contradiction with \eqref{non-degx0}. Similarly, if $\mu < \overline{\mu} < \infty$, then it follows from \Cref{double-lem} $(i)$ (see also \eqref{h-scale}) that
\[
h^{\overline{X}}(r) \le Cr^{2\overline{\mu}},
\]
and we have again reached a contradiction with \eqref{non-degx0}. Finally, if $\overline{\mu} = \infty$, fix $M > \mu$; then by an application of \Cref{double-infty} we find a constant $C$ such that 
\[
h^{\overline{X}}(r) \le Cr^{2M}
\]
holds for every $r$ sufficiently small. Reasoning as above, we reach the desired conclusion.
\end{proof}

For $X_0 \in S_{\mu}(u)$, $\mu < \infty$, the \emph{homogeneous rescalings} of $u$ and $v$ at $X_0$ (defined analogously to \eqref{hbu-def}) are given by 
\[
u_{\mu, r}^{X_0}(X) \coloneqq \frac{u(X_0 + rX)}{r^{\mu}}, \qquad v_{\mu, r}^{X_0}(X) \coloneqq \frac{v(X_0 + rX)}{r^{\mu}}.
\]
Similarly, we use $\tilde{u}_{\mu}^{X_0}$ and $\tilde{v}_{\mu}^{X_0}$ to denote the corresponding blow-up limits (see \Cref{HBUL}). 
Finally, the Monneau functional centered at $X_0$ is defined as 
\[
M_{\mu}^{X_0}(r, u, v, p, q) \coloneqq M_{\mu}(r, u(X_0 + \cdot), v(X_0 + \cdot), p(\cdot), q(\cdot)).
\]
With these definitions at hand, we proceed to prove that the blow-up limits $\tilde{u}_{\mu}^{X_0}$ and $\tilde{v}_{\mu}^{X_0}$ depend continuously on $X_0 \in S_{\mu}(u)$.

\begin{prop}
\label{mu-diff}
Let $u$ be a weak solution to \eqref{bil-noBC} in the sense of \Cref{ws-def}. Let $b \le 0$, $\mu$ as in \eqref{frequency_def}, $n = 2, 3$, and suppose $f$ satisfies \eqref{H2'}, \eqref{H4'}, and \eqref{H5'}. Assume that $u$ is not identically zero. Then, for every $\mu < \infty$, we have that the map $X_0 \mapsto (\tilde{u}_{\mu}^{X_0}, \tilde{v}_{\mu}^{X_0})$ is continuous from $S_{\mu}(u)$ to $\mathcal{P}_{\mu} \times \mathcal{P}_{\mu}$. Moreover, for every compact $K \subset S_{\mu}(u)$
\[
\|u_{\mu, r}^{X_0} - \tilde{u}_{\mu}^{X_0}\|_{L^{\infty}(B_{1/2})} + \|v_{\mu, r}^{X_0} - \tilde{v}_{\mu}^{X_0}\|_{L^{\infty}(B_{1/2})} \to 0
\]
as $r \to 0^+$, uniformly with respect to $X_0 \in S_{\mu}(u)$. In particular, there exists a modulus of continuity $\sigma_K$ such that 
\begin{align*}
|u(X) - \tilde{u}_{\mu}^{X_0}(X - X_0)| & \le \sigma_K(|X - X_0|)|X - X_0|^{\mu}, \\
|v(X) - \tilde{v}_{\mu}^{X_0}(X - X_0)| & \le \sigma_K(|X - X_0|)|X - X_0|^{\mu}, 
\end{align*}
for all $X_0 \in S_{\mu}(u)$ and all $X \in \overline{B_1^+} \setminus \overline{S_1^+}$. Finally, the result continues to hold for $n \ge 4$, provided that $\mu < 4/(n - 3)$.
\end{prop}
\begin{proof} We divide the proof into two steps.
\newline
\textbf{Step 1:} Let $\e > 0$ be given and fix $X_0 \in S_{\mu}(u)$. Recalling that $M_{\mu}^{X_0}(0^+, u, v, \tilde{u}_{\mu}^{X_0}, \tilde{v}_{\mu}^{X_0}) = 0$ (see, for example, \eqref{M1=0}), we can find $r_{\e} = r_{\e}(X_0)$ such that 
\[
M_{\mu}^{X_0}(r, u, v, \tilde{u}_{\mu}^{X_0}, \tilde{v}_{\mu}^{X_0}) < \e
\]
for all $r \le (0, r_{\e}]$. Using the fact that $u$ and $v$ are H\"older continuous with exponent $\alpha$, there exists $\rho_{\e} = \rho_{\e}(X_0, \mu, \alpha)$ such that for every $\overline{X} \in S_{\mu}(u)$ with $|\overline{X} - X_0| < \rho_{\e}$, we have that
\begin{align*}
M_{\mu}^{\overline{X}}(r_{\e}, u, v, \tilde{u}_{\mu}^{X_0}, \tilde{v}_{\mu}^{X_0})  
< 3\e.
\end{align*}
By \Cref{AlmostMonneautonicity}, assuming without loss that $r_{\e} < \e$, we have that 
\begin{equation}
\label{MX}
M_{\mu}^{\overline{X}}(r, u, v, \tilde{u}_{\mu}^{X_0}, \tilde{v}_{\mu}^{X_0}) \le M_{\mu}^{\overline{X}}(r_{\e}, u, v, \tilde{u}_{\mu}^{X_0}, \tilde{v}_{\mu}^{X_0}) + C(r_{\e} - r) \le C\e
\end{equation}
holds for all $r \in (0, r_{\e})$. In particular, taking the limit as $r \to 0^+$ in the identity
\[
M_{\mu}^{\overline{X}}(r, u, v, \tilde{u}_{\mu}^{X_0}, \tilde{v}_{\mu}^{X_0}) = \int_{S_1^+} y^b\left((u_{\mu, r}^{\overline{X}} - \tilde{u}_{\mu}^{X_0})^2 + (u_{\mu, r}^{\overline{X}} - \tilde{v}_{\mu}^{X_0})^2\right)\,d\Hh^n,
\]
in view of \eqref{MX} we obtain that
\begin{equation}
\label{X0X}
\int_{S_1^+} y^b\left((\tilde{u}_{\mu}^{\overline{X}} - \tilde{u}_{\mu}^{X_0})^2 + (\tilde{v}_{\mu}^{\overline{X}} - \tilde{v}_{\mu}^{X_0})^2\right)\,d\Hh^n = \lim_{r \to 0^+} M_{\mu}^{\overline{X}}(r, u, v, \tilde{u}_{\mu}^{X_0}, \tilde{v}_{\mu}^{X_0}) \le C\e.
\end{equation}
Since $\e$ is arbitrary, this proves the first part of the statement. 
\newline
\textbf{Step 2:} Next, let $K$ be a compact subset of  $S_{\mu}(u)$ and fix $X_0 \in K$. We claim that if $r_{\e}$ and $\rho_{\e}$ are chosen as in the previous step, then there exists a constant $C$ such that for $r \in (0, r_{\e})$ and every $\overline{X}$ with $|\overline{X} - X_0| < \rho_{\e}$ we have that 
\begin{equation}
\label{L2X}
\int_{B_1^+} y^b\left((u_{\mu, r}^{\overline{X}} - \tilde{u}_{\mu}^{\overline{X}})^2 + (v_{\mu, r}^{\overline{X}} - \tilde{v}_{\mu}^{\overline{X}})^2\right)\,dX \le C\e.
\end{equation}
Indeed, by the coarea formula together with a change of variables, we have that
\begin{align}
& \int_{B_1^+} y^b\left((u_{\mu, r}^{\overline{X}} - \tilde{u}_{\mu}^{\overline{X}})^2 + (v_{\mu, r}^{\overline{X}} - \tilde{v}_{\mu}^{\overline{X}})^2\right)\,dX \notag \\
& = \int_0^1 \int_{S_{\rho}^+} y^b\left[\left(\frac{u(\overline{X} + rX)}{r^{\mu}} - \tilde{u}_{\mu}^{\overline{X}}(X)\right)^2  + \left(\frac{v(\overline{X} + rX)}{r^{\mu}} - \tilde{v}_{\mu}^{\overline{X}}(X)\right)^2\right]\,d\Hh^n d\rho \notag \\
& = \int_0^1 \frac{1}{r^{n + b + 2\mu}} \int_{S_{\rho r}^+} y^b\left[(u(\overline{X} + X) - \tilde{u}_{\mu}^{\overline{X}}(X))^2  + (v(\overline{X} + X) - \tilde{v}_{\mu}^{\overline{X}}(X))^2\right]\,d\Hh^n d\rho \notag \\
& \le \int_0^1 2 \rho^{n + b + 2\mu} \left(M_{\mu}^{\overline{X}}(\rho r, u, v, \tilde{u}_{\mu}^{X_0}, \tilde{v}_{\mu}^{X_0}) +\int_{S_1^+} y^b\left((\tilde{u}_{\mu}^{\overline{X}} - \tilde{u}_{\mu}^{X_0})^2 + (\tilde{v}_{\mu}^{\overline{X}} - \tilde{v}_{\mu}^{X_0})^2\right)\,d\Hh^n \right)\,d\rho. \label{L2Xuni}
\end{align}
As one can readily check, \eqref{L2X} follows by substituting \eqref{MX} and \eqref{X0X} into \eqref{L2Xuni}. Next, observe that for any $\overline{X} = (\bar{x}, 0) \in S_{\mu}(u)$ we have that
\begin{equation}
\label{uvX}
\left\{
\arraycolsep=1.4pt\def\arraystretch{1.6}
\begin{array}{rll}
\Delta_b \left(u_{\mu, r}^{\overline{X}} - \tilde{u}_{\mu}^{\overline{X}}\right) = & r^2v_{\mu, r}^{\overline{X}} & \text{ in } B_1^+, \\
\Delta_b \left(v_{\mu, r}^{\overline{X}} - \tilde{v}_{\mu}^{\overline{X}}\right) = & 0 & \text{ in } B_1^+, \\
\partial_y^b \left(u_{\mu, r}^{\overline{X}} - \tilde{u}_{\mu}^{\overline{X}}\right) = & 0 & \text{ on } B_1', \\
\partial_y^b \left(v_{\mu, r}^{\overline{X}} - \tilde{v}_{\mu}^{\overline{X}}\right) = &  \bar{f}(\cdot, u_{\mu, r}^{\overline{X}}) & \text{ on } B_1',
\end{array}
\right.
\end{equation}
where 
\[
\bar{f}(x, u_{\mu, r}^{\overline{X}}(x, 0)) \coloneqq r^{1 - b - \mu} f(\bar{x} + rx, r^{\mu}u_{\mu, r}^{\overline{X}}(x, 0)).
\]
For any $q > (n + 1)/2$, an application of \Cref{STV} yields the existence of a positive constant $C$ and $\alpha \in (0, 1)$ such that
\begin{equation}
\label{u-tay}
\|u_{\mu, r}^{\overline{X}} - \tilde{u}_{\mu}^{\overline{X}}\|_{C^{0, \alpha}(B_{3/4}^+)} \le C\left(\|u_{\mu, r}^{\overline{X}} - \tilde{u}_{\mu}^{\overline{X}}\|_{L^2(B_1^+; y^b)} + \|r^2 v_{\mu, r}^{\overline{X}}\|_{L^{q}(B_1^+; y^b)}\right).
\end{equation}
By inspection in the proof of \Cref{double-lem} $(i)$, we can find $C$ such that
\[
h^{\overline{X}}(r) \le C r^{2\mu} 
\]
for all $\overline{X} \in S_{\mu}(u) \in \overline{B_{R}^+}$, $R < 1$. In turn, we have that
\begin{align}
\|r^2 v_{\mu, r}^{\overline{X}}\|_{L^q(B_1^+; y^b)}^q & = r^2 \int_{B_1^+} y^b |v_{\mu, r}^{\overline{X}}|^q\,dX \notag \\
& = \frac{r^2}{r^{n + 1 + b + q\mu}} \int_0^r \int_{S_{\rho}^+} y^b |v(\overline{X} + X)|^q\,d\Hh^nd\rho \notag \\
& \le \frac{Cr^2}{r^{n + 1 + b + q\mu}} \int_0^r \rho^{n + b} h^{\overline{X}}(\rho)\,d\rho \notag \\
& \le \frac{Cr^2}{r^{n + 1 + b + q\mu}} \int_0^r \rho^{n + b + 2\mu}\,d\rho \le Cr^{2 - \mu(q - 2)}, \label{ptw-fix}
\end{align}
where $C$ is a constant that depends on $\|v\|_{L^{\infty}(B_R^+)}^{q - 2}$ when $q > 2$. Observe that the right-hand side of \eqref{ptw-fix} can be made arbitrarily small, provided that $2 - \mu(q - 2) > 0$. In particular, if $n = 2$ then $q = 2$ has all the required properties and the desired result readily follows, while if $n = 3$, we can choose 
\[
q \coloneqq 2 + \frac{1}{1 + \mu} > 2.
\]
Finally, if $n \ge 4$, since by assumption we have that $\mu < 4/(n - 3)$, then we can find $q \in ((n + 1)/2, 2 + 2/\mu)$. This shows that there exists $C = C(\e)$, independent of $\overline{X}$ and $r$ such that $C(\e) \to 0^+$ as $\e \to 0^+$ and with the property that
\begin{equation}
\label{uXunif}
\|u_{\mu, r}^{\overline{X}} - \tilde{u}_{\mu}^{\overline{X}}\|_{L^{\infty}(B_{3/4}^+)} \le C(\e)
\end{equation}
for all $r \in (0, r_{\e})$. 
Similarly, applying \Cref{STV1.5} for the equation satisfied by $v_{\mu, r}^{\overline{X}} - \tilde{v}_{\mu}^{\overline{X}}$ (see  \eqref{uvX}), we obtain that
\begin{equation}
\label{v-tay}
\|v_{\mu, r}^{\overline{X}} - \tilde{v}_{\mu}^{\overline{X}}\|_{L^{\infty}(B_{1/2}^+)} \le C\left(\|v_{\mu, r}^{\overline{X}} - \tilde{v}_{\mu}^{\overline{X}}\|_{L^2(B_{3/4}^+; y^b)} + \|\bar{f}(\cdot, u_{\mu, r}^{\overline{X}})\|_{L^{\infty}(B_{3/4}')}\right).
\end{equation}
Notice also that by \eqref{H2'} there exists a constant $C$ (independent of $r$ and $\overline{X}$) such that
\begin{equation}
\label{barH2}
|\bar{f}(x, u_{\mu, r}^{\overline{X}}(x, 0))| \le C r^{1 - b + \mu(p - 2)} |u_{\mu, r}^{\overline{X}}(x, 0)|^{p - 1}.
\end{equation}
In particular, we have that
\begin{multline}
\label{fboundX}
\|\bar{f}(\cdot, u_{\mu, r}^{\overline{X}})\|_{L^{\infty}(B_{3/4}')} \le C r^{1 - b + \mu(p - 2)} \Big(\|u_{\mu, r}^{\overline{X}} - \tilde{u}_{\mu}^{\overline{X}}\|_{L^{\infty}(B_{3/4}')} \\ + \|\tilde{u}_{\mu}^{\overline{X}} - \tilde{u}_{\mu}^{X_0}\|_{L^{\infty}(B_{3/4}')} + \|\tilde{u}_{\mu}^{X_0}\|_{L^{\infty}(B_{3/4}')}\Big)^{p - 1}.
\end{multline}
Observe that the first term on the right-hand side of \eqref{fboundX} is bounded by \eqref{uXunif}, and that the second term in controlled by \eqref{X0X} in view of the fact that all norms are equivalent on the finite dimensional space $\mathcal{P}_{\mu}$; finally, the last term is bounded by a constant that depends only on $X_0$. 

Since $K$ is compact, the desired result follows by a covering argument. This completes the proof.
\end{proof}

Next, we define the dimension of the singular set $S_{\mu}(u)$ at a point $X_0$ as 
\begin{equation}
\label{dim-def}
d_{\mu}^{X_0} \coloneqq \min \left\{d_{u, \mu}^{X_0}, d_{v, \mu}^{X_0}\right\},
\end{equation}
where 
\[
d_{u, \mu}^{X_0} \coloneqq \operatorname{dim}\{ \xi \in \RR^n : \xi \cdot \nabla_x' \tilde{u}_{\mu}^{X_0}(x, 0) = 0 \text{ for all } x \in \RR^n\}
\]
and 
\[
d_{v, \mu}^{X_0} \coloneqq \operatorname{dim}\{ \xi \in \RR^n : \xi \cdot \nabla_x' \tilde{v}_{\mu}^{X_0}(x, 0) = 0 \text{ for all } x \in \RR^n\}.
\]

\begin{thm}
\label{GPthm}
Under the assumptions of \Cref{mu-diff}, let 
\begin{equation}
\label{Sdmu}
S_{\mu}^d(u) \coloneqq \{X_0 \in S_{\mu}(u) : d_{\mu}^{X_0} = d \}.
\end{equation}
Then for every $\mu \in \NN$ we have that $d_{\mu}^{X_0} \le n - 1$. Moreover, for every $d \in \{0, 1, \dots, n - 1\}$, the set $S_{\mu}^d(u)$ is contained in a countable union of $d$-dimensional $C^1$ manifolds.
\end{thm}
\begin{proof}
Observe that since $\mu \in \NN$, we must have that $u \equiv v \equiv 0$ on $S_{\mu}(u)$. We begin by proving that $d_{\mu}^{X_0} \le n - 1$. To this end, arguing by contradiction, assume that $d_{\mu}^{X_0} = n$. In particular, this assumption implies that $\nabla_x' \tilde{u}_{\mu}^{X_0}$ and $\nabla_x' \tilde{v}_{\mu}^{X_0}$ vanish identically on $\RR^n \times \{0\}$. Consequently, the $\mathcal{L}_b$-harmonic functions $\tilde{u}_{\mu}^{X_0}$ and $\tilde{v}_{\mu}^{X_0}$ must also vanish identically on $\RR^n \times \{0\}$. Applying Proposition 2.2 in \cite{MR3268922}, we deduce that $\tilde{u}_{\mu}^{X_0}$ and $\tilde{v}_{\mu}^{X_0}$ are identically zero in $\RR^{n + 1}$, which contradicts \Cref{HBUL}.

Next, notice that by \Cref{mu-diff} we have that 
\begin{align*}
|\tilde{u}_{\mu}^{X_0}(X - X_0)| & \le \sigma(|X - X_0|)|X - X_0|^{\mu}, \\
|\tilde{v}_{\mu}^{X_0}(X - X_0)| & \le \sigma(|X - X_0|)|X - X_0|^{\mu} 
\end{align*}
for all $X, X_0 \in S_{\mu}(u)$. Assume first that $d_{\mu}^{X_0} = d_{u, \mu}^{X_0}$ and observe that, if this is the case, then $\tilde{u}_{\mu}^{X_0}$ is necessarily non-trivial. The desired result then readily follows by arguing as in the proof of Theorem 1.3.8 in \cite{MR2511747}. On the other hand, if $d_{\mu}^{X_0} = d_{v, \mu}^{X_0}$, then can repeat the same argument for the non-trivial function $\tilde{v}_{\mu}^{X_0}$. This concludes the proof.
\end{proof}

\begin{rmk}
We conclude the section with remarks and open problems.
\begin{itemize}
\item[$(i)$] It is worth noting that the structure of the set $S_1(u)$ for the fourth-order problem under consideration may differ significantly from its second-order counterpart, where points with frequency $\mu = 1$ are all regular. Indeed, in our case, we must still consider the scenario where $\mu = 1$ and $\nabla_x' u(X_0) = 0$, thereby proving that $\tilde{u}_{\mu}^{X_0} \equiv 0$. This, in turn, implies that $\tilde{v}_{\mu}^{X_0}$ cannot vanish identically. We notice also that $\tilde{v}_{\mu}^{X_0}$ is independent of $y$ as a consequence of the Neumann boundary condition on $B_1'$ when $b \le 0$.
\item[$(ii)$] In view of \Cref{GPthm}, it remains to analyze the structure of the singular sets $S_0(u)$ and $S_{\infty}(u)$. Observe that if $\mu = 0$, we have $d_{\mu}^{X_0} = n$. Therefore, $S_0(u) = S_0^n(u)$ is trivially contained in a manifold of dimension $n$. Therefore, while \Cref{GPthm} in a sense continues to hold also for $\mu = 0$, it does not provide any valuable information on the structure of $S_0(u)$. 
Similarly, whether the value $\mu = \infty$ is admissible remains unclear, and the analysis of $S_{\infty}(u)$ is entirely open, likely requiring very different techniques.
\end{itemize}
\end{rmk}

\section{Conflict of interest and data availability statements}
On behalf of all authors, the corresponding author declares that there is no conflict of interest. This manuscript has no associated data.

\section*{Acknowledgements}
We would like to express our sincere gratitude to the anonymous reviewer whose remarks and suggestions substantially improved the quality of this manuscript.

\bibliographystyle{siam}
\bibliography{thin-bL}
\end{document}